\documentclass[letterpaper,11pt,twoside,keywordsasfootnote,addressatend,noinfoline]{article}
\usepackage{fullpage}
\usepackage[english]{babel}
\usepackage{amssymb}
\usepackage{amsmath}
\usepackage{theorem}
\usepackage{epsfig}
\usepackage{subfigure}
\usepackage{mathrsfs}
\usepackage{imsart}
\usepackage{color}

\theorembodyfont{\slshape}
\newtheorem{theorem}{Theorem}

\newtheorem{lemma}[theorem]{Lemma}
\newtheorem{corollary}[theorem]{Corollary}

\newenvironment{proof}{\noindent{\scshape Proof.}}{\hspace*{2mm}~$\square$}
\newenvironment{demo}[1]{\noindent{\textbf{Proof of #1}}}{\hspace*{2mm}~$\square$}

\newcommand{\Z}{\mathbb{Z}}
\newcommand{\R}{\mathbb{R}}
\newcommand{\N}{\mathbb{N}}

\newcommand{\ind}{\mathbf{1}}
\newcommand{\ep}{\epsilon}
\newcommand{\threshold}{\tau}
\newcommand{\GG}{\mathscr G}
\newcommand{\VG}{\mathscr V}
\newcommand{\EG}{\mathscr E}
\newcommand{\OO}{\Gamma}
\newcommand{\VO}{V}
\newcommand{\EO}{E}
\newcommand{\weight}{W}
\newcommand{\Weight}{\mathbf W}
\newcommand{\radius}{\mathbf{r}}
\newcommand{\diameter}{\mathbf{d}}
\newcommand{\integer}[1]{\lfloor{#1}\rfloor}
\newcommand{\ceil}[1]{\lceil{#1}\rceil}

\DeclareMathOperator{\card}{card \,}
\DeclareMathOperator{\binomial}{Binomial \,}
\DeclareMathOperator{\geometric}{Geometric \,}
\DeclareMathOperator{\hypergeometric}{Hypergeometric \,}
\DeclareMathOperator{\cont}{cont \,}
\DeclareMathOperator{\cent}{center}
\DeclareMathOperator{\reg}{reg}

\begin{document}
\begin{frontmatter}
\title     {Limiting behavior for a general class of \\ voter models with confidence threshold}
\runtitle  {Voter models with confidence threshold}
\author    {Nicolas Lanchier\thanks{Research partially supported by NSF Grant DMS-10-05282.} and
            Stylianos Scarlatos\thanks{Research partially supported through Scholarship of Excellence 2014-2015 jointly by the Greek State Scholarships Foundation (IKY) and Siemens.}}
\runauthor {N. Lanchier and S. Scarlatos}
\address   {School of Mathematical and Statistical Sciences \\ Arizona State University \\ Tempe, AZ 85287, USA.}
\address   {Department of Biomedical Engineering \\ and Computational Science \\ Aalto University \\ Aalto FI-00076, Finland.}

\maketitle

\begin{abstract} \ \
 This article is concerned with a general class of stochastic spatial models for the dynamics of opinions.
 Like in the voter model, individuals are located on the vertex set of a connected graph and update their opinion at a constant rate
 based on the opinion of their neighbors.
 However, unlike in the voter model, the set of opinions is represented by the set of vertices of another connected graph that we
 call the opinion graph:
 when an individual interacts with a neighbor, she imitates this neighbor if and only if the distance between their opinions, defined
 as the graph distance induced by the opinion graph, does not exceed a certain confidence threshold.
 When the confidence threshold is at least equal to the radius of the opinion graph, we prove that the one-dimensional process
 fluctuates and clusters and give a universal lower bound for the probability of consensus of the process on finite connected graphs.
 We also establish a general sufficient condition for fixation of the infinite system based on the structure of the opinion graph,
 which we then significantly improve for opinion graphs which are distance-regular.
 Our general results are used to understand the dynamics of the system for various examples of opinion graphs:
 paths and stars, which are not distance-regular, and cycles, hypercubes and the five Platonic solids, which are distance-regular.
\end{abstract}



\end{frontmatter}

%%%%%%%%%%%%%%%%%%%%%%%%%%%%%%%%%%%%%%%%%%%%%%%%%%%%%%%%%%%%%%%%%%%%%%%%%%%%%%%%%%%%%%%%%%%%%%%%%%%%%%%%%%%%%%%%%%%%%%%%%%%%%%%%%%%%%%%%%%

\section{Introduction}
\label{sec:intro}

\indent Since the work of Arratia~\cite{arratia_1983} on annihilating random walks, it is known that, when starting with infinitely
 many supporters of each opinion, the one-dimensional voter model fluctuates, i.e., the number of opinion changes at each vertex is
 almost surely infinite.
 In contrast, as a consequence of irreducibility, the process on finite connected graphs fixates to a configuration in which all
 the vertices share the same opinion.
 The objective of this paper is to study the dichotomy between fluctuation and fixation for a general class of opinion models
 with confidence threshold.
 The main novelty is to equip the set of opinions with the structure of a connected graph and use the induced graph distance to
 define mathematically a level of disagreement among individuals.
 Based on this modeling approach, some of the most popular models of opinion dynamics can be recovered by choosing the
 structure of the opinion space suitably:
 the constrained voter model, independently introduced in~\cite{itoh_etal_1998, vazquez_krapivsky_redner_2003}, is obtained by assuming
 that the opinion space is a path, while the Axelrod model for the dissemination of cultures~\cite{axelrod_1997} and the discrete
 Deffuant model~\cite{deffuant_al_2000} are closely related to our models when the opinion space is a Hamming graph
 and a hypercube, respectively. \vspace*{8pt}

% % % % % % % % % % % % % % % % % % % % % % % % % % % % % % % % % % % % % % % % % % % % % % % % % % % % % % % % % % % % % % % % % % % % % % %

\noindent{\bf Model description} --
 The class of models considered in this article are examples of interacting particle systems inspired from the voter
 model~\cite{clifford_sudbury_1973, holley_liggett_1975} for the dynamics of opinions.
 Individuals are located on the vertex set of a connected graph and characterized by their opinion, with the set of opinions being
 identified with the vertex set of another connected graph.
 The former graph represents the underlying spatial structure and is used to determine the interaction neighborhood of each individual.
 The latter graph, that we call the opinion graph, represents the structure of the opinion space and is used to determine the distance
 between two opinions and the level of disagreement between two individuals.
 From now on, we call spatial distance the graph distance induced by the spatial structure and opinion distance the graph distance induced
 by the opinion graph.
 Individuals interact with each of their neighbors at rate one.
 As the result of an interaction, an individual imitates her neighbor if and only if the distance between their opinions just before the
 interaction does not exceed some confidence threshold~$\threshold \in \N$.
 More formally, we let
 $$ \begin{array}{rclcl}
    \GG & := & (\VG, \EG) & = & \hbox{the {\bf spatial structure}} \vspace*{2pt} \\
    \OO & := & (\VO, \EO) & = & \hbox{the {\bf opinion graph}} \end{array} $$
 be two connected graphs, where~$\OO$ is also assumed to be finite.
 Then, our opinion model is the continuous-time Markov chain whose state at time~$t$ is a spatial configuration
 $$ \eta_t : \VG \ \longrightarrow \ \VO \quad \hbox{where} \quad \eta_t (x) = \hbox{opinion at~$x \in \VG$ at time~$t$} $$
 and with transition rates at vertex~$x \in \VG$ given by
\begin{equation}
\label{eq:rates}
  \begin{array}{rrl}
      c_{i \to j} (x, \eta) & := & \lim_{h \to 0} \,(1/h) \,P \,(\eta_{t + h} (x) = j \,| \,\eta_t = \eta \ \hbox{and} \ \eta (x) = i) \vspace*{4pt} \\
                            &  = & \card \{y \in N_x : \eta (y) = j \} \ \ind \{d (i, j) \leq \threshold \} \quad \hbox{for all} \quad i, j \in V. \end{array}
\end{equation}
 Here, the set~$N_x$ denotes the interaction neighborhood of vertex~$x$, i.e., all the vertices which are at spatial distance one
 from~$x$, while~$d (i, j)$ refers to the opinion distance between~$i$ and~$j$, which is the length of the shortest path connecting
 both opinions on the opinion graph.
 Note that the classical voter model is simply obtained by assuming that the opinion graph consists of two vertices
 connected by an edge and that the confidence threshold equals one.
 The general class of opinion models described by the transition rates~\eqref{eq:rates} where the opinion space is represented by
 a finite connected graph equipped with its graph distance has been recently introduced in~\cite{scarlatos_2013}. \vspace*{8pt}

% % % % % % % % % % % % % % % % % % % % % % % % % % % % % % % % % % % % % % % % % % % % % % % % % % % % % % % % % % % % % % % % % % % % % % %

\noindent{\bf Main results} --
 The main question about the general model is whether the system fluctuates and clusters, leading ultimately the population to a global
 consensus, or fixates in a highly fragmented configuration.
 Recall that the process is said to
\begin{itemize}
 \item {\bf fluctuate} when~$P \,(\eta_t (x) \ \hbox{changes infinitely often}) = 1$ for all~$x \in \VG$, \vspace*{3pt}
 \item {\bf fixate} when~$P \,(\eta_t (x) \ \hbox{changes a finite number of times}) = 1$ for all~$x \in \VG$, \vspace*{3pt}
 \item {\bf cluster} when~$P \,(\eta_t (x) = \eta_t (y)) \to 1$ as~$t \to \infty$ for all~$x, y \in \VG$.
\end{itemize}
 Note that whether the system fluctuates and clusters or fixates in a fragmented configuration is very sensitive to the initial configuration.
 Also, throughout this paper, we assume that the process starts from a product measure with densities which are constant across space, i.e.,
 $$ \rho_j \ := \ P \,(\eta_0 (x) = j) \quad \hbox{for all} \quad (x, j) \in \VG \times \VO $$
 only depends on opinion~$j$ but not on site~$x$.
 To avoid trivialities, these densities are assumed to be positive.
 Sometimes, we will make the stronger assumption that all the opinions are equally likely at time zero.
 These two hypotheses correspond to the following two conditions:
\begin{align}
 \label{eq:product} \rho_j \ > \ 0  \hspace*{15pt} \quad \hbox{for all} \quad j \in V \vspace*{3pt} \\
 \label{eq:uniform} \rho_j \ = \ F^{-1}            \quad \hbox{for all} \quad j \in V
\end{align}
 where~$F := \card V$ refers to the total number of opinions.
 Key quantities to understand the long-term behavior of the system are the radius and the diameter of the opinion graph defined respectively
 as the minimum and maximum eccentricity of any vertex:
 $$ \begin{array}{rclcl}
        \radius & := & \min_{i \in \VO} \ \max_{j \in \VO} \ d (i, j) & = & \hbox{the {\bf radius} of the graph~$\OO$} \vspace*{3pt} \\
      \diameter & := & \max_{i \in \VO} \ \max_{j \in \VO} \ d (i, j) & = & \hbox{the {\bf diameter} of the graph~$\OO$}. \end{array} $$
 To state our first theorem, we also introduce the subset
\begin{equation}
\label{eq:center}
  C (\OO, \threshold) \ := \ \{i \in \VO : d (i, j) \leq \threshold \ \hbox{for all} \ j \in \VO \}
\end{equation}
 that we shall call the~{\bf $\threshold$-center} of the opinion graph.
 The next result states that, whenever the confidence threshold is at least equal to the radius of the opinion graph, the infinite one-dimensional
 system fluctuates and clusters while the probability that the finite system reaches ultimately a consensus, i.e., fixates in a configuration where
 all the individuals share the same opinion, is bounded from below by a positive constant that does not depend on the size of the spatial structure.
 Here, infinite one-dimensional means that the spatial structure is the graph with vertex set~$\Z$ and where each vertex is connected to its
 two nearest neighbors.
\begin{theorem} --
\label{th:fluctuation}
 Assume~\eqref{eq:product}. Then,
\begin{enumerate}
 \item[a.] the process on~$\Z$ fluctuates whenever
 \begin{equation}
 \label{eq:fluctuation}
    d (i, j) \leq \threshold \quad \hbox{for all} \quad (i, j) \in V_1 \times V_2 \quad \hbox{for some $V$-partition~$\{V_1, V_2 \}$}.
 \end{equation}
\end{enumerate}
 Assume in addition that~$\radius \leq \threshold$. Then,
\begin{enumerate}
 \item[b.] the process on~$\Z$ clusters and \vspace*{3pt}
 \item[c.] the probability of consensus on any finite connected graph satisfies
  $$ \begin{array}{l} P \,(\eta_t \equiv \hbox{constant for some} \ t > 0) \ \geq \ \rho_{\cent} := \sum_{j \in C (\OO, \threshold)} \,\rho_j \ > \ 0. \end{array} $$
\end{enumerate}
\end{theorem}
 We will show that the~$\threshold$-center is nonempty if and only if the threshold is at least equal to the radius so the
 probability of consensus in the last part is indeed positive.
 In fact, except when the threshold is at least equal to the diameter, in which case all three conclusions of the theorem turn out to be trivial,
 when the threshold is at least equal to the radius, both the~$\threshold$-center and its complement are nonempty, and therefore form a partition
 that satisfies~\eqref{eq:fluctuation}.
 In particular, fluctuation also holds when the radius is not more than the threshold.
 We also point out that the last part of the theorem implies that the average domain length in the final absorbing state scales like the population
 size, namely~$\card \VG$.
 This result applies in particular to the constrained voter model where the opinion graph is a path with three vertices interpreted as leftists,
 centrists and rightists, thus contradicting the conjecture on domain length scaling in~\cite{vazquez_krapivsky_redner_2003}.

\indent We now seek for sufficient conditions for fixation of the infinite one-dimensional system, beginning with general opinion graphs.
 At least for the process starting from the uniform product measure, these conditions can be expressed using
 $$ N (\OO, s) \ := \ \card \{(i, j) \in \VO \times \VO : d (i, j) = s \} \quad \hbox{for} \quad s = 1, 2, \ldots, \diameter, $$
 which is the number of pairs of opinions at opinion distance~$s$ of each other.
 In the statement of the next theorem, the function~$\ceil{\,\cdot \,}$ refers to the ceiling function.
\begin{theorem} --
\label{th:fixation}
 For the opinion model on~$\Z$, fixation occurs
\begin{enumerate}
 \item[a.] when~\eqref{eq:uniform} holds and
 \begin{equation}
 \label{eq:th-fixation}
   \begin{array}{l} S (\OO, \threshold) \ := \ \sum_{k > 0} \,((k - 2) \,\sum_{s : \ceil{s / \threshold} = k} \,N (\OO, s)) \ > \ 0, \end{array}
 \end{equation}
 \item[b.] for some initial distributions~\eqref{eq:product} when~$\diameter > 2 \threshold$.
\end{enumerate}
\end{theorem}
 Combining Theorems~\ref{th:fluctuation}.a and~\ref{th:fixation}.b shows that these two results are sharp when~$\diameter = 2 \radius$,
 which holds for opinion graphs such as paths and stars:
 for such graphs, the one-dimensional system fluctuates starting from any initial distribution~\eqref{eq:product}
 if and only if~$\radius \leq \threshold$.

\indent Our last theorem, which is also the most challenging result of this paper, gives a significant improvement of the previous condition for fixation for {\bf distance-regular} opinion graphs.
 This class of graphs is defined mathematically as follows: let
 $$ \OO_s (i) \ := \ \{j \in \VO : d (i, j) = s \} \quad \hbox{for} \quad s = 0, 1, \ldots, \diameter $$
 be the {\bf distance partition} of the vertex set~$\VO$ for some~$i \in \VO$.
 Then, the opinion graph is said to be a distance-regular graph when the so-called {\bf intersection numbers}
\begin{equation}
\label{eq:dist-reg-1}
  \begin{array}{rrl}
      N (\OO, (i_-, s_-), (i_+, s_+)) & := & \card (\OO_{s_-} (i_-) \cap \OO_{s_+} (i_+)) \vspace*{3pt} \\
                                      &  = & \card \{j \in \VO : d (i_-, j) = s_- \ \hbox{and} \ d (i_+, j) = s_+ \} \vspace*{3pt} \\
                                      &  = & f (s_-, s_+, d (i_-, i_+)) \end{array}
\end{equation}
 only depend on the distance~$d (i_-, i_+)$ but not on the particular choice of~$i_-$ and~$i_+$.
 This implies that, for distance-regular opinion graphs, the number of vertices
 $$ N (\OO, (i, s)) \ := \ \card (\OO_s (i)) \ = \ f (s, s, 0) \ =: \ h (s) $$
 does not depend on vertex~$i$.
 To state our last theorem, we let
 $$ \begin{array}{l} \Weight (k) \ := \ - 1 + \sum_{1 < n \leq k} \,\sum_{n \leq m \leq \ceil{\diameter / \threshold}} \,(q_n \,q_{n + 1} \cdots q_{m - 1}) / (p_n \,p_{n + 1} \cdots p_m) \end{array} $$
 where by convention an empty sum is equal to zero and an empty product is equal to one, and where the coefficients~$p_n$ and~$q_n$ are defined in terms of the intersection numbers as
 $$ \begin{array}{rcl}
      p_n & := & \max \,\{\sum_{s : \ceil{s / \threshold} = n - 1} f (s_-, s_+, s) / h (s_+) : \ceil{s_- / \threshold} = 1 \ \hbox{and} \ \ceil{s_+ / \threshold} = n \} \vspace*{3pt} \\
      q_n & := & \,\min \,\{\sum_{s : \ceil{s / \threshold} = n + 1} f (s_-, s_+, s) / h (s_+) : \ceil{s_- / \threshold} = 1 \ \hbox{and} \ \ceil{s_+ / \threshold} = n \}. \end{array} $$
 Then, we have the following sufficient condition for fixation.
\begin{theorem} --
\label{th:dist-reg}
 Assume~\eqref{eq:uniform} and~\eqref{eq:dist-reg-1}.
 Then, the process on~$\Z$ fixates when
\begin{equation}
\label{eq:th-dist-reg}
  \begin{array}{l} S_{\reg} (\OO, \threshold) \ := \ \sum_{k > 0} \,(\Weight (k) \,\sum_{s : \ceil{s / \threshold} = k} \,h (s)) \ > \ 0. \end{array}
\end{equation}
\end{theorem}
 To understand the coefficients~$p_n$ and~$q_n$, we note that, letting~$i_-$ and~$j$ be two opinions at opinion distance~$s_-$ of each other, we have the following
 interpretation:
 $$ \begin{array}{rcl}
      f (s_-, s_+, s) / h (s_+) & = & \hbox{probability that an opinion~$i_+$ chosen uniformly} \vspace*{0pt} \\
                                &   & \hbox{at random among the opinions at distance~$s_+$ from} \\
                                &   & \hbox{opinion~$j$ is at distance~$s$ from opinion~$i_-$}. \end{array} $$

% % % % % % % % % % % % % % % % % % % % % % % % % % % % % % % % % % % % % % % % % % % % % % % % % % % % % % % % % % % % % % % % % % % % % % %

\noindent{\bf Outline of the proofs} --
 The lower bound for the probability of consensus on finite connected graphs follows from the optional stopping theorem after proving that
 the process that keeps track of the number of supporters of opinions belonging to the~$\threshold$-center is a martingale.
 The analysis of the infinite system is more challenging.
 The first key to all our proofs is to use the formal machinery introduced in~\cite{lanchier_2012, lanchier_scarlatos_2013, lanchier_schweinsberg_2012}
 that consists in keeping track of the disagreements along the edges of the spatial structure.
 This technique has also been used in~\cite{lanchier_moisson_2014, lanchier_scarlatos_2014} to study related models.
 In the context of our general opinion model, we put a pile of~$s$ particles on edges that connect individuals who are at opinion distance~$s$
 of each other, i.e., we set
 $$ \xi_t ((x, x + 1)) \ := \ d (\eta_t (x), \eta_t (x + 1)) \quad \hbox{for all} \quad x \in \Z. $$
 The definition of the confidence threshold implies that, piles with at most~$\threshold$ particles, that we call active, evolve
 according to symmetric random walks, while larger piles, that we call frozen, are static.
 In addition, the jump of an active pile onto another pile results in part of the particles being annihilated.
 The main idea to prove fluctuation is to show that, after identifying opinions that belong to the same member of the
 partition~\eqref{eq:fluctuation}, the process reduces to the voter model, and use that the one-dimensional voter model
 fluctuates according to~\cite{arratia_1983}.
 Fluctuation, together with the stronger assumption~$\radius \leq \threshold$, implies that the frozen piles, and ultimately all the piles
 of particles, go extinct, which is equivalent to clustering of the opinion model.

\indent In contrast, fixation occurs when the frozen piles have a positive probability of never being reduced, which is more difficult to establish.
 To briefly explain our approach to prove fixation, we say that the pile at~$(x, x + 1)$ is of order~$k$ when
 $$ (k - 1) \,\threshold \ < \ \xi_t ((x, x + 1)) \ \leq \ k \threshold. $$
 To begin with, we use a construction due to~\cite{bramson_griffeath_1989} to obtain an implicit condition for fixation in terms
 of the initial number of piles of any given order in a large interval.
 Large deviation estimates for the number of such piles are then proved and used to turn this implicit condition into the explicit
 condition~\eqref{eq:th-fixation}.
 To derive this condition, we use that at least~$k - 1$ active piles must jump onto a pile initially of order~$k > 1$ to turn this pile
 into an active pile.
 Condition~\eqref{eq:th-fixation} is obtained assuming the worst case scenario when the number of particles that annihilate is maximal.
 To show the improved condition for fixation~\eqref{eq:th-dist-reg} for distance-regular opinion graphs, we use the same approach but
 count more carefully the number of annihilating events.
 First, we use duality-like techniques to prove that, when the opinion graph is distance-regular, the system of piles becomes Markov.
 This is used to prove that the jump of an active pile onto a pile of order~$n > 1$ reduces/increases its order with respective
 probabilities at most~$p_n$ and at least~$q_n$.
 This implies that the number of active piles that must jump onto a pile initially of order~$k > 1$ to turn it into an active pile
 is stochastically larger than the first hitting time to state~1 of a certain discrete-time birth and death process.
 This hitting time is equal in distribution to
 $$ \begin{array}{l} \sum_{1 < n \leq k} \,\sum_{n \leq m \leq \ceil{\diameter / \threshold}} \,(q_n \,q_{n + 1} \cdots q_{m - 1}) / (p_n \,p_{n + 1} \cdots p_m) \ = \ 1 + \Weight (k). \end{array} $$
 The probabilities~$p_n$ and~$q_n$ are respectively the death parameter and the birth parameter of the discrete-time birth and death
 process while the integer~$\ceil{\diameter / \threshold}$ is the number of states of this process, which is also the maximum order
 of a pile. \vspace*{8pt}

% % % % % % % % % % % % % % % % % % % % % % % % % % % % % % % % % % % % % % % % % % % % % % % % % % % % % % % % % % % % % % % % % % % % % % %

\noindent{\bf Application to concrete opinion graphs} --
\begin{figure}[t]
\centering
\includegraphics[width=0.98\textwidth]{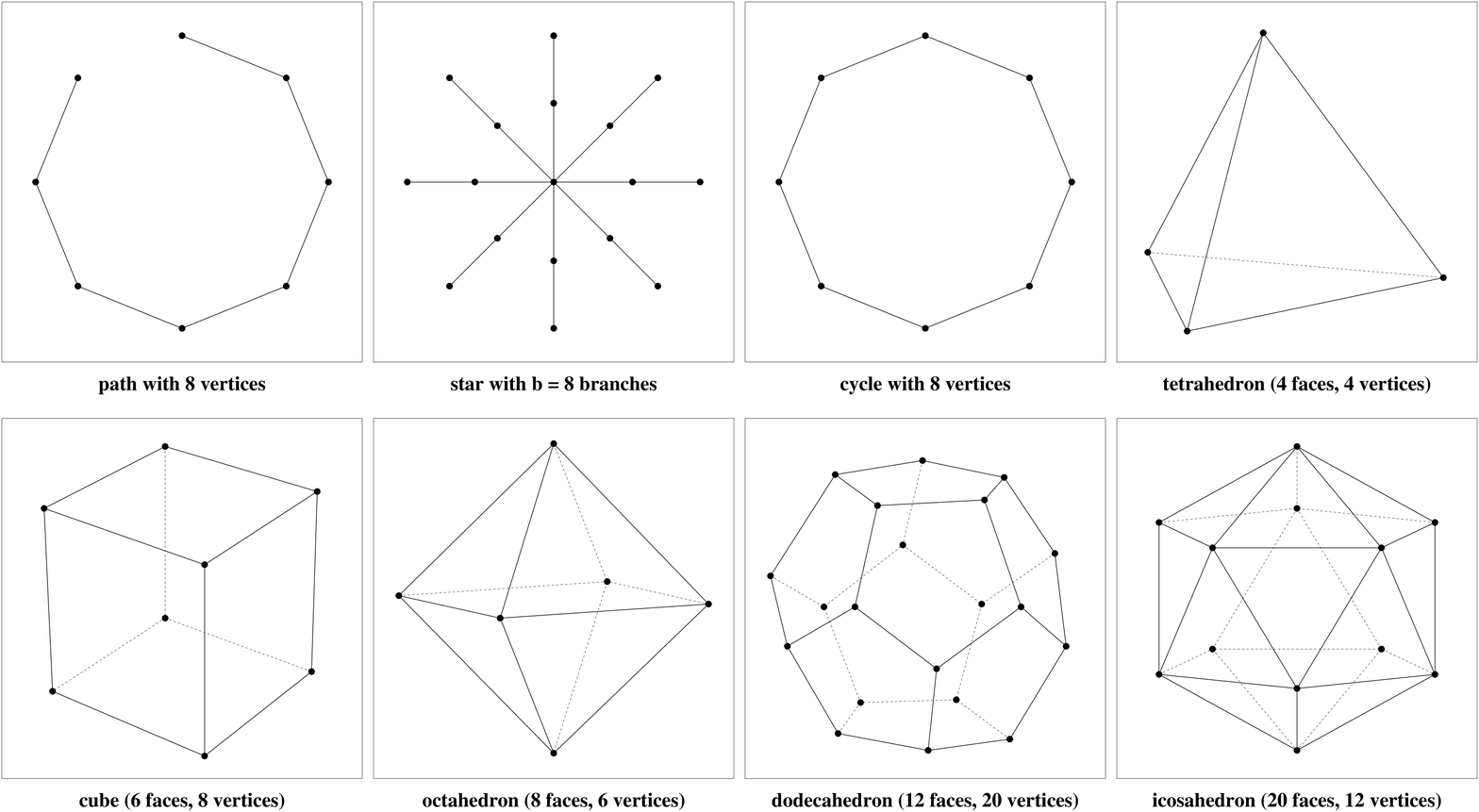}
\caption{\upshape{Opinion graphs considered in Corollaries~\ref{cor:path}--\ref{cor:hypercube}}}
\label{fig:graphs}
\end{figure}
 We now apply our general results to particular opinion graphs, namely the ones which are represented in Figure~\ref{fig:graphs}.
 First, we look at paths and more generally stars with~$b$ branches of equal length.
 For paths, one can think of the individuals as being characterized by their position about one issue, ranging from strongly agree to strongly disagree.
 For stars, individuals are offered~$b$ alternatives:
 the center represents undecided individuals while vertices far from the center are more extremist in their position.
 These graphs are not distance-regular so we can only apply Theorem~\ref{th:fixation} to study fixation of the infinite system.
 This theorem combined with Theorem~\ref{th:fluctuation} gives the following two corollaries.
\begin{corollary}[path] --
\label{cor:path}
 When~$\OO$ is the path with~$F$ vertices,
\begin{itemize}
 \item the system fluctuates when~\eqref{eq:product} holds and~$F \leq 2 \threshold + 1$ whereas \vspace*{3pt}
 \item the system fixates when~\eqref{eq:uniform} holds and~$3 F^2 - (20 \threshold + 3) \,F + 10 \,(3 \threshold + 1) \,\threshold > 0$.
\end{itemize}
\end{corollary}
\begin{corollary}[star] --
\label{cor:star}
 When~$\OO$ is the star with~$b$ branches of length~$r$,
\begin{itemize}
 \item the system fluctuates when~\eqref{eq:product} holds and~$r \leq \threshold$ whereas \vspace*{3pt}
 \item the system fixates when~\eqref{eq:uniform} holds, $2r > 3 \threshold$ and
 $$ 4 \,(b - 1) \,r^2 + 2 \,((4 - 5b) \,\threshold + b - 1) \,r + (6b - 5) \,\threshold^2 + (1 - 2b) \,\threshold \ > \ 0. $$
\end{itemize}
\end{corollary}
 To illustrate Theorem~\ref{th:dist-reg}, we now look at distance-regular graphs, starting with the five convex regular polyhedra also known as the Platonic solids.
 These graphs are natural mathematically though we do not have any specific interpretation from the point of view of social sciences except, as explained below,
 for the cube and more generally hypercubes.
 For these five graphs, Theorems~\ref{th:fluctuation} and~\ref{th:dist-reg} give sharp results with the exact value of the critical threshold except
 for the dodecahedron for which the behavior when~$\threshold = 3$ remains an open problem.
\begin{corollary}[Platonic solids] --
\label{cor:polyhedron}
 Assume~\eqref{eq:uniform}. Then,
\begin{itemize}
 \item the tetrahedral model fluctuates for all~$\threshold \geq 1$, \vspace*{2pt}
 \item the cubic model fluctuates when~$\threshold \geq 2$ and fixates when~$\threshold \leq 1$, \vspace*{2pt}
 \item the octahedral model fluctuates for all~$\threshold \geq 1$, \vspace*{2pt}
 \item the dodecahedral model fluctuates when~$\threshold \geq 4$ and fixates when~$\threshold \leq 2$, \vspace*{2pt}
 \item the icosahedral model fluctuates when~$\threshold \geq 2$ and fixates when~$\threshold \leq 1$.
\end{itemize}
\end{corollary}
 Next, we look at the case where the individuals are characterized by some preferences represented by the set of vertices of a cycle.
 For instance, as explained in~\cite{boudourides_scarlatos_2005}, all strict orderings of three alternatives can be represented by the cycle
 with~$3! = 6$ vertices.
\begin{corollary}[cycle] --
\label{cor:cycle}
 When~$\OO$ is the cycle with~$F$ vertices,
\begin{itemize}
 \item the system fluctuates when~\eqref{eq:product} holds and~$F \leq 2 \threshold + 2$ whereas \vspace*{3pt}
 \item the system fixates when~\eqref{eq:uniform} hold and~$F \geq 4 \threshold + 2$.
\end{itemize}
\end{corollary}
 Finally, we look at hypercubes with~$F = 2^d$ vertices, which are generalizations of the three-dimensional cube.
 In this case, the individuals are characterized by their position --~in favor or against~-- about~$d$ different issues, and the opinion distance between two
 individuals is equal to the number of issues they disagree on.
 Theorem~\ref{th:dist-reg} gives the following result.
\begin{corollary}[hypercube] --
\label{cor:hypercube}
 When~$\OO$ is the hypercube  with~$2^d$ vertices,
\begin{itemize}
 \item the system fluctuates when~\eqref{eq:product} holds and~$d \leq \threshold + 1$ whereas \vspace*{3pt}
 \item the system fixates when~\eqref{eq:uniform} holds and~$d / \threshold > 3$ or when~$d / \threshold > 2$ with~$\threshold$ large.
\end{itemize}
\end{corollary}
\begin{table}[t]
\begin{center}
\begin{tabular}{cccccc}
\hline \noalign{\vspace*{2pt}}
 opinion graph   & radius                     & diameter                     & fluctuation                & fix. ($\threshold = 1$)  & fix. ($\threshold$ large)                              \\ \noalign{\vspace*{1pt}} \hline \noalign{\vspace*{6pt}}
 path            & $\radius = \integer{F/2}$  & $\diameter = F - 1$          & $F \leq 2 \threshold + 1$  & $F \geq 6$               & $F / \threshold > (10 + \sqrt{10}) / 3  \approx 4.39$  \\ \noalign{\vspace*{3pt}}
 star~($b = 3$)  & $\radius = r$              & $\diameter = 2r$             & $r \leq \threshold$        & $r \geq 2$               & $r / \threshold > (11 + \sqrt{17}) / 8  \approx 1.89$  \\ \noalign{\vspace*{3pt}}
 star~($b = 5$)  & $\radius = r$              & $\diameter = 2r$             & $r \leq \threshold$        & $r \geq 2$               & $r / \threshold > (21 + \sqrt{41}) / 16 \approx 1.71$  \\ \noalign{\vspace*{3pt}}
 cycle           & $\radius = \integer{F/2}$  & $\diameter = \integer{F/2}$  & $F \leq 2 \threshold + 2$  & $F \geq 6$               & $F / \threshold > 4$                                   \\ \noalign{\vspace*{3pt}}
 hypercube       & $\radius = d$              & $\diameter = d$              & $d \leq \threshold + 1$    & $d \geq 3$               & $d / \threshold > 2$                                   \\ \noalign{\vspace*{8pt}} \hline \noalign{\vspace*{3pt}}
 opinion graph   & radius                     & diameter                     & fluctuation                & \multicolumn{2}{c}{fixation when}                                                 \\ \noalign{\vspace*{1pt}} \hline \noalign{\vspace*{6pt}}
 tetrahedron     & $\radius = 1$              & $\diameter = 1$              & $\threshold \geq 1$        & \multicolumn{2}{c}{$\threshold = 0$}                                              \\ \noalign{\vspace*{3pt}}
 cube            & $\radius = 3$              & $\diameter = 3$              & $\threshold \geq 2$        & \multicolumn{2}{c}{$\threshold \leq 1$}                                           \\ \noalign{\vspace*{3pt}}
 octahedron      & $\radius = 2$              & $\diameter = 2$              & $\threshold \geq 1$        & \multicolumn{2}{c}{$\threshold = 0$}                                              \\ \noalign{\vspace*{3pt}}
 dodecahedron    & $\radius = 5$              & $\diameter = 5$              & $\threshold \geq 4$        & \multicolumn{2}{c}{$\threshold \leq 2$}                                           \\ \noalign{\vspace*{3pt}}
 icosahedron     & $\radius = 3$              & $\diameter = 3$              & $\threshold \geq 2$        & \multicolumn{2}{c}{$\threshold \leq 1$}                                           \\ \noalign{\vspace*{4pt}} \hline
\end{tabular}
\end{center}
\caption{\upshape{Summary of our results for the opinion graphs in Figure~\ref{fig:graphs}}}
\label{tab:summary}
\end{table}
 Table~\ref{tab:summary} summarizes our results for the graphs of Figure~\ref{fig:graphs}.
 The second and third columns give the value of the radius and the diameter.
 The conditions in the fourth column are the conditions for fluctuation of the infinite system obtained from the corollaries.
 For opinion graphs with a variable number of vertices, the last two columns give sufficient conditions for fixation in the two extreme cases when
 the confidence threshold is one and when the confidence threshold is large.
 To explain the last column for paths and stars, note that the opinion model fixates whenever~$\diameter / \threshold$
 is larger than the largest root of the polynomials
 $$ \begin{array}{rl}
      3 X^2 - 20 X + 30 & \hbox{for the path} \vspace*{3pt} \\
      2 X^2 - 11 X + 13 & \hbox{for the star with~$b = 3$ branches} \vspace*{3pt} \\
      4 X^2 - 21 X + 25 & \hbox{for the star with~$b = 5$ branches} \end{array} $$
 and the diameter of the opinion graph is sufficiently large.
 These polynomials are obtained from the conditions in Corollaries~\ref{cor:path}--\ref{cor:star} by only keeping the terms with degree two.

%%%%%%%%%%%%%%%%%%%%%%%%%%%%%%%%%%%%%%%%%%%%%%%%%%%%%%%%%%%%%%%%%%%%%%%%%%%%%%%%%%%%%%%%%%%%%%%%%%%%%%%%%%%%%%%%%%%%%%%%%%%%%%%%%%%%%%%%%%

\section{Coupling with a system of annihilating particles}
\label{sec:coupling}

\indent To study the one-dimensional system, it is convenient to construct the process from a graphical representation and to introduce
 a coupling between the opinion model and a certain system of annihilating particles that keeps track of the discrepancies along the edges
 of the lattice rather than the opinion at each vertex.
 This system of particles can also be constructed from the same graphical representation.
 Since the opinion model on general finite graphs will be studied using other techniques, we only define the graphical representation
 for the process on~$\Z$, which consists of the following collection of independent Poisson processes:
\begin{itemize}
 \item For each~$x \in \Z$, we let~$(N_t (x, x \pm 1) : t \geq 0)$ be a rate one Poisson process. \vspace*{3pt}
 \item We denote by~$T_n (x, x \pm 1) := \inf \,\{t : N_t (x, x \pm 1) = n \}$ its~$n$th arrival time.
\end{itemize}
 This collection of independent Poisson processes is then turned into a percolation structure by drawing an arrow~$x \to x \pm 1$
 at time~$t := T_n (x, x \pm 1)$.
 We say that this arrow is {\bf active} when
 $$ d (\eta_{t-} (x), \eta_{t-} (x \pm 1)) \ \leq \ \threshold. $$
 The configuration at time~$t$ is then obtained by setting
\begin{equation}
\label{eq:rule}
  \begin{array}{rcll}
   \eta_t (x \pm 1) & = & \eta_{t-} (x) & \hbox{when the arrow~$x \to x \pm 1$ is active} \vspace*{3pt} \\
                    & = & \eta_{t-} (x \pm 1) & \hbox{when the arrow~$x \to x \pm 1$ is not active} \end{array}
\end{equation}
 and leaving the opinion at all the other vertices unchanged.
 An argument due to Harris \cite{harris_1972} implies that the opinion model starting from any configuration can indeed
 be constructed using this percolation structure and rule~\eqref{eq:rule}.
 From the collection of active arrows, we construct active paths as in percolation theory.
 More precisely, we say that there is an {\bf active path} from~$(z, s)$ to~$(x, t)$, and write~$(z, s) \leadsto (x, t)$, whenever there exist
 $$ s_0 = s < s_1 < \cdots < s_{n + 1} = t \qquad \hbox{and} \qquad
    x_0 = z, \,x_1, \,\ldots, \,x_n = x $$
 such that the following two conditions hold:
\begin{itemize}
 \item For~$j = 1, 2, \ldots, n$, there is an active arrow~$x_{j - 1} \to x_j$ at time~$s_j$. \vspace*{3pt}
 \item For~$j = 0, 1, \ldots, n$, there is no active arrow that points at~$\{x_j \} \times (s_j, s_{j + 1})$.
\end{itemize}
 These two conditions imply that
 $$ \hbox{for all} \ (x, t) \in \Z \times \R_+ \ \hbox{there is a unique} \ z \in \Z \ \hbox{such that} \ (z, 0) \leadsto (x, t). $$
 Moreover, because of the definition of active arrows, the opinion at vertex~$x$ at time~$t$ originates from and is therefore equal to the
 initial opinion at vertex~$z$ so we call vertex~$z$ the {\bf ancestor} of vertex~$x$ at time~$t$.

\indent As previously mentioned, to study the one-dimensional system, we look at the process that keeps track of the discrepancies along the edges
 rather than the actual opinion at each vertex, that we shall call the {\bf system of piles}.
 To define this process, it is convenient to identify each edge with its midpoint and to define translations on the edge set as follows:
 $$ \begin{array}{rclcl}
          e & := & \{x, x + 1 \} \ \equiv \ x + 1/2         & \hbox{for all} & x \in \Z \vspace*{3pt} \\
      e + v & := & \{x, x + 1 \} + v \ \equiv \ x + 1/2 + v & \hbox{for all} & (x, v) \in \Z \times \R. \end{array} $$
 The system of piles is then defined as
 $$ \xi_t (e) \ := \ d (\eta_t (e - 1/2), \eta_t (e + 1/2)) \quad \hbox{for all} \quad e \in \Z + 1/2, $$
 and it is convenient to think of edge~$e$ as being occupied by a pile of~$\xi_t (e)$ particles.
 The dynamics of the opinion model induces the following evolution rules on this system of particles.
 Assuming that there is an arrow~$x - 1 \to x$ at time~$t$ and that
 $$ \begin{array}{rcl}
    \xi_{t-} (x - 1/2) & := & d (\eta_{t-} (x), \eta_{t-} (x - 1)) \ = \ s_- \vspace*{3pt} \\
    \xi_{t-} (x + 1/2) & := & d (\eta_{t-} (x), \eta_{t-} (x + 1)) \ = \ s_+ \end{array} $$
 we have the following alternative:
\begin{itemize}
 \item In case~$s_- = 0$, meaning that there is no particle on the edge, the two interacting agents already agree just before the interaction therefore nothing happens. \vspace*{3pt}
 \item In case~$s_- > \threshold$, meaning that there are more than~$\threshold$ particles on the edge, the two interacting agents disagree too much to trust each other so nothing happens. \vspace*{3pt}
 \item In case~$0 < s_- \leq \threshold$, meaning that there is at least one but no more than~$\threshold$ particles on the edge, the agent at vertex~$x$ mimics her left neighbor, which gives
   $$ \begin{array}{rcl}
      \xi_t (x - 1/2) & := & d (\eta_t (x), \eta_t (x - 1)) \ = \ d (\eta_{t-} (x - 1), \eta_{t-} (x - 1)) \ = \ 0 \vspace*{3pt} \\
      \xi_t (x + 1/2) & := & d (\eta_t (x), \eta_t (x + 1)) \ = \ d (\eta_{t-} (x - 1), \eta_{t-} (x + 1)). \end{array} $$
   In particular, there is no more particles at edge~$x - 1/2$.
   In addition, the size~$s$ of the pile of particles at edge~$x + 1/2$ at time~$t$, where size of a pile refers to the number of particles
   in that pile, satisfies the two inequalities
\begin{equation}
\label{eq:size}
    \begin{array}{rcl}
      s & \leq & |d (\eta_{t-} (x - 1), \eta_{t-} (x)) + d (\eta_{t-} (x), \eta_{t-} (x + 1))| \ = \ |s_- + s_+| \vspace*{3pt} \\
      s & \geq & |d (\eta_{t-} (x - 1), \eta_{t-} (x)) - d (\eta_{t-} (x), \eta_{t-} (x + 1))| \ = \ |s_- - s_+|. \end{array}
\end{equation}
   Note that the first inequality implies that the process involves deaths of particles but no births, which is a key property that will be used later.
\end{itemize}
 Similar evolution rules are obtained by exchanging the direction of the interaction from which we deduce the following description
 for the dynamics of piles:
\begin{itemize}
 \item Piles with more than~$\threshold$ particles cannot move: we call such piles {\bf frozen piles} and the particles in such piles frozen particles. \vspace*{3pt}
 \item Piles with at most~$\threshold$ particles jump one unit to the left or to the right at rate one: we call such piles {\bf active piles} and the particles in such piles active particles.
  Note that arrows in the graphical representation are active if and only if they cross an active pile. \vspace*{3pt}
 \item When a pile of size~$s_-$ jumps onto a pile of size~$s_+$ this results in a pile whose size~$s$ satisfies the two inequalities in~\eqref{eq:size}
  so we say that~$s_- + s_+ - s$ particles are {\bf annihilated}.
\end{itemize}

%%%%%%%%%%%%%%%%%%%%%%%%%%%%%%%%%%%%%%%%%%%%%%%%%%%%%%%%%%%%%%%%%%%%%%%%%%%%%%%%%%%%%%%%%%%%%%%%%%%%%%%%%%%%%%%%%%%%%%%%%%%%%%%%%%%%%%%%%%

\section{Proof of Theorem~\ref{th:fluctuation}}
\label{sec:fluctuation}

\indent Before proving the theorem, we start with some preliminary remarks.
 To begin with, we observe that, when the diameter~$\diameter \leq \threshold$, the~$\threshold$-center covers all the opinion graph,
 indicating that the model reduces to a multitype voter model with~$F = \card \VO$ opinions.
 In this case, all three parts of the theorem are trivial, with the probability of consensus in the last part being equal to one.
 To prove the theorem in the nontrivial case~$\threshold < \diameter$, we introduce the set
\begin{equation}
\label{eq:boundary}
  B (\OO, \threshold) \ := \ \{i \in \VO : d (i, j) > \threshold \ \hbox{for some} \ j \in \VO \}
\end{equation}
 and call this set the~{\bf $\threshold$-boundary} of the opinion graph.
 One key ingredient to our proof is the following lemma, which gives a sufficient condition for~\eqref{eq:fluctuation} to hold.
\begin{lemma} --
\label{lem:partition}
 The sets~$V_1 = C (\OO, \threshold)$ and~$V_2 = B (\OO, \threshold)$ satisfy~\eqref{eq:fluctuation} when~$\radius \leq \threshold < \diameter$.
\end{lemma}
\begin{proof}
 From~\eqref{eq:center} and~\eqref{eq:boundary}, we get~$B (\OO, \threshold) = \VO \setminus C (\OO, \threshold)$ therefore
 $$ C (\OO, \threshold) \,\cup \,B (\OO, \threshold) = V \quad \hbox{and} \quad C (\OO, \threshold) \,\cap \,B (\OO, \threshold) = \varnothing. $$
 In addition, the~$\threshold$-center of the graph is nonempty because
\begin{equation}
\label{eq:center-radius}
  \begin{array}{rcl}
    C (\OO, \threshold) \neq \varnothing & \hbox{if and only if} & \hbox{there is~$i \in \VO$ such that~$d (i, j) \leq \threshold$ for all~$j \in \VO$} \vspace*{3pt} \\
                                         & \hbox{if and only if} & \hbox{there is~$i \in \VO$ such that~$\max_{j \in \VO} \,d (i, j) \leq \threshold$} \vspace*{3pt} \\
                                         & \hbox{if and only if} & \min_{i \in \VO} \,\max_{j \in \VO} \,d (i, j) \leq \threshold \vspace*{3pt} \\
                                         & \hbox{if and only if} & \radius \leq \threshold \end{array}
\end{equation}
 while the~$\threshold$-boundary is nonempty because
 $$ \begin{array}{rcl}
      B (\OO, \threshold) \neq \varnothing & \hbox{if and only if} & \hbox{there is~$i \in \VO$ such that~$d (i, j) > \threshold$ for some~$j \in \VO$} \vspace*{3pt} \\
                                           & \hbox{if and only if} & \hbox{there are~$i, j \in \VO$ such that~$d (i, j) > \threshold$} \vspace*{3pt} \\
                                           & \hbox{if and only if} & \max_{i \in \VO} \,\max_{j \in \VO} \,d (i, j) > \threshold \vspace*{3pt} \\
                                           & \hbox{if and only if} & \diameter > \threshold. \end{array} $$
 This shows that~$\{V_1, V_2 \}$ is a partition of the set of opinions.
 Finally, since all the vertices in the~$\threshold$-center are within distance~$\threshold$ of all the other vertices, condition~\eqref{eq:fluctuation} holds.
\end{proof} \\ \\
 This lemma will be used in the proof of part b where clustering will follow from fluctuation, and in the proof of part c to show that
 the probability of consensus on any finite connected graph is indeed positive.
 From now on, we call vertices in the~$\threshold$-center the centrist opinions and vertices in the~$\threshold$-boundary the extremist opinions. \\ \\
\begin{demo}{Theorem~\ref{th:fluctuation}a (fluctuation)} --
 Under condition~\eqref{eq:fluctuation}, agents who support an opinion in the set~$V_1$ are within the confidence threshold of
 agents who support an opinion in~$V_2$, therefore we deduce from the expression of the transition rates~\eqref{eq:rates} that
\begin{equation}
\label{eq:fluctuation-1}
  \begin{array}{rcl}
     c_{i \to j} (x, \eta_t) & = &
          \lim_{h \to 0} \ (1/h) \,P \,(\eta_{t + h} (x) = j \,| \,\eta_t \ \hbox{and} \ \eta_t (x) = i) \vspace*{4pt} \\ & = &
          \card \{y \in N_x :  \eta_t (y) = j \} \end{array}
\end{equation}
 for every~$(i, j) \in V_1 \times V_2$ and every~$(i, j) \in V_2 \times V_1$. Let
\begin{equation}
\label{eq:fluctuation-3}
  \zeta_t (x) \ := \ \ind \{\eta_t (x) \in V_2 \} \quad \hbox{for all} \quad x \in \Z.
\end{equation}
 Since, according to~\eqref{eq:fluctuation-1}, we have
\begin{itemize}
 \item for all~$j \in V_2$, the rates~$c_{i \to j} (x, \eta_t)$ are constant across all~$i \in V_1$, \vspace*{4pt}
 \item for all~$i \in V_1$, the rates~$c_{j \to i} (x, \eta_t)$ are constant across all~$j \in V_2$,
\end{itemize}
 the process~$(\zeta_t)$ is Markov with transition rates
 $$ \begin{array}{rrl}
     c_{0 \to 1} (x, \zeta_t) & := &
          \lim_{h \to 0} \ (1/h) \,P \,(\zeta_{t + h} (x) = 1 \,| \,\zeta_t \ \hbox{and} \ \zeta_t (x) = 0) \vspace*{4pt} \\ & = &
          \sum_{i \in V_1} \sum_{j \in V_2} c_{i \to j} (x, \eta_t) \,P \,(\eta_t (x) = i \,| \,\zeta_t (x) = 0) \vspace*{4pt} \\ & = &
          \sum_{i \in V_1} \sum_{j \in V_2} \card \{y \in N_x :  \eta_t (y) = j \} \,P \,(\eta_t (x) = i \,| \,\zeta_t (x) = 0) \vspace*{4pt} \\ & = &
          \sum_{j \in V_2} \card \{y \in N_x : \eta_t (y) = j \} \ = \ \card \{y \in N_x : \zeta_t (y) = 1 \} \end{array} $$
 and similarly for the reverse transition
 $$ \begin{array}{l}
     c_{1 \to 0} (x, \zeta_t) \ = \ \card \{y \in N_x : \eta_t (y) \in V_1 \} \ = \ \card \{y \in N_x : \zeta_t (y) = 0 \}. \end{array} $$
 This shows that~$(\zeta_t)$ is the voter model.
 In addition, since~$V_1, V_2 \neq \varnothing$,
 $$ \begin{array}{l} P \,(\zeta_0 (x) = 0) \ = \ P \,(\eta_0 (x) \in V_1) \ = \ \sum_{j \in V_1} \,\rho_j \ \in \ (0, 1) \end{array} $$
 whenever condition~\eqref{eq:product} holds.
 In particular, the lemma follows from the fact that the one-dimensional voter model starting with a positive density of each type fluctuates.
 This last result is a consequence of site recurrence for annihilating random walks proved in \cite{arratia_1983}.
\end{demo} \\ \\
\begin{demo}{Theorem~\ref{th:fluctuation}b (clustering)} --
 Since~$\radius \leq \threshold < \diameter$,
 $$ V_1 \ = \ C (\OO, \threshold) \quad \hbox{and} \quad V_2 \ = \ B (\OO, \threshold) $$
 form a partition of~$\VO$ according to Lemma~\ref{lem:partition}.
 This implies in particular that, not only the opinion model fluctuates, but also the coupled voter model~\eqref{eq:fluctuation-3}
 for this specific partition fluctuates, which is the key to the proof.
 To begin with, we define the function
 $$ \begin{array}{l} u (t) \ := \ E \,\xi_t (e) \ = \ \sum_{0 \leq j \leq \diameter} \,j \,P \,(\xi_t (e) = j) \end{array} $$
 which, in view of translation invariance of the initial configuration and the evolution rules, does not depend on the choice of~$e$.
 Note that, since the system of particles coupled with the process involves deaths of particles but no births, the function~$u (t)$ is nonincreasing in time.
 Since it is also nonnegative, it has a limit:~$u (t) \to l$ as~$t \to \infty$.
 Now, on the event that an edge~$e$ is occupied by a pile with at least one particle at a given time~$t$, we have the following alternative:
\begin{itemize}
 \item[(1)] In case edge~$e := x + 1/2$ carries a frozen pile, since the centrist agents are within the confidence threshold of all the other individuals, we must have
  $$ \eta_t (x) \in V_2 = B (\OO, \threshold) \quad \hbox{and} \quad \eta_t (x + 1) \in V_2 = B (\OO, \threshold). $$
  Now, using that the voter model~\eqref{eq:fluctuation-3} fluctuates,
  $$ T \ := \ \inf \,\{s > t : \eta_s (x) \in V_1 = C (\OO, \threshold) \ \hbox{or} \ \eta_s (x + 1) \in V_1 = C (\OO, \threshold) \} < \infty $$
  almost surely, while by definition of the~$\threshold$-center, we have
  $$ \xi_T (e) \ = \ d (\eta_T (x), \eta_T (x + 1)) \ \leq \ \threshold \ < \ \xi_t (e). $$
  In particular, at least one of the frozen particles at~$e$ is annihilated eventually. \vspace*{4pt}
 \item[(2)] In case edge~$e := x + 1/2$ carries an active pile, since one-dimensional symmetric random walks are recurrent, this pile eventually intersects
  another pile.
  Let~$s_-$ and~$s_+$ be respectively the size of these two piles and let~$s$ be the size of the pile of particles resulting from their intersection.
  Then, we have the following alternative:
 \begin{itemize}
  \item[(a)] In case~$s < s_- + s_+$ and~$s > \threshold$, at least one particle is annihilated and there is either formation or increase of a frozen pile so we are
   back to case~(1): since the voter model coupled with the opinion model fluctuates, at least one of the frozen particles in this pile is annihilated eventually.
  \item[(b)] In case~$s < s_- + s_+$ and~$s \leq \threshold$, at least one particle is annihilated.
  \item[(c)] In case~$s = s_- + s_+$ and~$s > \threshold$, there is either formation or increase of a frozen pile so we are back to case~(1):
   since the voter model coupled with the opinion model fluctuates, at least one of the frozen particles in this pile is annihilated eventually.
  \item[(d)] In case~$s = s_- + s_+$ and~$s \leq \threshold$, the resulting pile is again active so it keeps moving until, after a finite number of collisions, we are back to either~(a) or~(b) or~(c)
   and at least one particle is annihilated eventually.
 \end{itemize}
\end{itemize}
 This shows that there is a sequence~$0 < t_1 < \cdots < t_n < \cdots < \infty$ such that
 $$ u (t_n) \ \leq \ (1/2) \,u (t_{n - 1}) \ \leq \ (1/4) \,u (t_{n - 2}) \ \leq \ \cdots \ \leq \ (1/2)^n \,u (0) \ \leq \ (1/2)^n \,F $$
 from which it follows that the density of particles decreases to zero:
 $$ \begin{array}{l} \lim_{t \to \infty} \,P \,(\xi_t (e) \neq 0) \ \leq \ \lim_{t \to \infty} \,u (t) \ = \ 0 \quad \hbox{for all} \quad e \in \Z + 1/2. \end{array} $$
 In conclusion, for all~$x, y \in \Z$ with~$x < y$, we have
 $$ \begin{array}{rcl}
    \lim_{t \to \infty} \,P \,(\eta_t (x) \neq \eta_t (y)) & \leq &
    \lim_{t \to \infty} \,P \,(\xi_t (z + 1/2) \neq 0 \ \hbox{for some} \ x \leq z < y) \vspace*{4pt} \\ & \leq &
    \lim_{t \to \infty} \,\sum_{x \leq z < y} \,P \,(\xi_t (z + 1/2) \neq 0) \vspace*{4pt} \\ & = &
     (y - x) \lim_{t \to \infty} \,P \,(\xi_t (e) \neq 0) \ = \ 0, \end{array} $$
 which proves clustering.
\end{demo} \\ \\
 The third part of the theorem, which gives a lower bound for the probability of consensus of the process on finite connected graphs,
 relies on very different techniques, namely techniques related to martingale theory following an idea from \cite{lanchier_2010}, section 3.
 However, the partition of the opinion set into centrist opinions and extremist opinions is again a key to the proof. \\ \\
\begin{demo}{Theorem~\ref{th:fluctuation}c (consensus)} --
 We first prove that the process that keeps track of the number of supporters of any given opinion is a martingale.
 Then, applying the optional stopping theorem, we obtain a lower bound for the probability of extinction of the extremist agents, which is
 also a lower bound for the probability of consensus.
 For every~$j \in \VO$, we set
 $$ X_t (j) \ := \ \card \{x \in \VG : \eta_t (x) = j \} \quad \hbox{and} \quad X_t \ := \ \card \{x \in \VG : \eta_t (x) \in C (\OO, \threshold) \} $$
 and we observe that
\begin{equation}
\label{eq:consensus-1}
 \begin{array}{l}
   X_t \ = \ \sum_{j \in C (\OO, \threshold)} X_t (j). \end{array}
\end{equation}
 Letting~$\mathcal F_t$ denote the natural filtration of the process, we also have
 $$ \begin{array}{l}
    \lim_{h \to 0} \ (1/h) \,E \,(X_{t + h} (j) - X_t (j) \,| \,\mathcal F_t) \vspace*{4pt} \\ \hspace*{25pt} = \
    \lim_{h \to 0} \ (1/h) \,P \,(X_{t + h} (j) - X_t (j) = 1 \,| \,\mathcal F_t) \vspace*{4pt} \\ \hspace*{40pt} - \
    \lim_{h \to 0} \ (1/h) \,P \,(X_{t + h} (j) - X_t (j) = - 1 \,| \,\mathcal F_t) \vspace*{4pt} \\ \hspace*{25pt} = \
    \card \{(x, y) \in \EG : \eta_t (x) \neq j \ \hbox{and} \ \eta_t (y) = j \ \hbox{and} \ |\eta_t (x) - j| \leq \threshold \} \vspace*{4pt} \\ \hspace*{40pt} - \
    \card \{(x, y) \in \EG : \eta_t (x) = j \ \hbox{and} \ \eta_t (y) \neq j \ \hbox{and} \ |\eta_t (y) - j| \leq \threshold \} \ = \ 0. \end{array} $$
 This shows that the process~$X_t (j)$ is a martingale with respect to the natural filtration of the opinion model.
 This, together with equation~\eqref{eq:consensus-1}, implies that~$X_t$ also is a martingale.
 Because of the finiteness of the graph, this martingale is bounded and gets trapped in an absorbing state after an almost surely
 finite stopping time:
 $$ T \ := \ \inf \,\{t : \eta_t = \eta_s \ \hbox{for all} \ s > t \} \ < \ \infty \quad \hbox{almost surely}. $$
 We claim that~$X_T$ can only take two values:
\begin{equation}
\label{eq:consensus-2}
  X_T \in \{0, N \} \quad \hbox{where} \quad \hbox{$N := \card (\VG)$ = the population size}.
\end{equation}
 Indeed, assuming by contradiction that~$X_T \notin \{0, N \}$ implies the existence of an absorbing state with at least one centrist
 agent and at least one extremist agent.
 Since the graph is connected, this further implies the existence of an edge~$e = (x, y)$ such that
 $$ \eta_T (x) \in C (\OO, \threshold) \quad \hbox{and} \quad \eta_T (y) \in B (\OO, \threshold) $$
 but then we have
 $$ \eta_T (x) \neq \eta_T (y) \quad \hbox{and} \quad d (\eta_T (y), \eta_T (x)) \leq \threshold $$
 showing that~$\eta_T$ is not an absorbing state, in contradiction with the definition of time~$T$.
 This proves that our claim~\eqref{eq:consensus-2} is true.
 Now, applying the optional stopping theorem to the bounded martingale~$X_t$ and the almost surely finite stopping time~$T$,
 we obtain
 $$ \begin{array}{l} E X_T \ = \ E X_0 \ = \ N \times P \,(\eta_0 (x) \in C (\OO, \threshold)) \ = \ N \times \sum_{j \in C (\OO, \threshold)} \,\rho_j \ = \ N \rho_{\cent} \end{array} $$
 which, together with~\eqref{eq:consensus-2}, implies that
\begin{equation}
\label{eq:consensus-3}
 \begin{array}{rcl}
   P \,(X_T = N) & = & (1/N)(0 \times P \,(X_T = 0) + N \times P \,(X_T = N)) \vspace*{4pt} \\
                 & = & (1/N) \ E X_T \ = \ \rho_{\cent}. \end{array}
\end{equation}
 To conclude, we observe that, on the event that~$X_T = N$, all the opinions present in the system at the time to absorption
 are centrist opinions and since the only absorbing states with only centrist opinions are the configurations in which all the agents
 share the same opinion, we deduce that the system converges to a consensus.
 This, together with~\eqref{eq:consensus-3}, implies that
 $$ P \,(\eta_t \equiv \hbox{constant for some} \ t > 0) \ \geq \ P \,(X_T = N) \ = \ \rho_{\cent}. $$
 Finally, since the threshold is at least equal to the radius, it follows from~\eqref{eq:center-radius} that
 the~$\threshold$-center is nonempty, so we have~$\rho_{\cent} > 0$.
 This completes the proof of Theorem~\ref{th:fluctuation}.
\end{demo}

%%%%%%%%%%%%%%%%%%%%%%%%%%%%%%%%%%%%%%%%%%%%%%%%%%%%%%%%%%%%%%%%%%%%%%%%%%%%%%%%%%%%%%%%%%%%%%%%%%%%%%%%%%%%%%%%%%%%%%%%%%%%%%%%%%%%%%%%%%

\section{Sufficient condition for fixation}
\label{sec:condition}

\indent This section and the next two ones are devoted to the proof of Theorem~\ref{th:fixation} which studies the fixation regime
 of the infinite one-dimensional system.
 In this section, we give a general sufficient condition for fixation that can be expressed based on the initial number of active
 particles and frozen particles in a large random interval.
 The main ingredient of the proof is a construction due to Bramson and Griffeath~\cite{bramson_griffeath_1989} based on
 duality-like techniques looking at active paths.
 The next section establishes large deviation estimates for the initial number of particles in order to simplify
 the condition for fixation using instead the expected number of active and frozen particles per edge.
 This is used in the subsequent section to prove Theorem~\ref{th:fixation}.
 The next lemma gives a condition for fixation based on properties of the active paths, which is the
 analog of~\cite[Lemma~2]{bramson_griffeath_1989}.
\begin{lemma} --
\label{lem:fixation-condition}
 For all~$z \in \Z$, let
 $$ T (z) \ := \ \inf \,\{t : (z, 0) \leadsto (0, t) \}. $$
 Then, the opinion model on~$\Z$ fixates whenever
\begin{equation}
\label{eq:fixation}
 \begin{array}{l} \lim_{N \to \infty} \,P \,(T (z) < \infty \ \hbox{for some} \ z < - N) \ = \ 0. \end{array}
\end{equation}
\end{lemma}
\begin{proof}
 This follows closely the proof of~\cite[Lemma~4]{lanchier_scarlatos_2013}.
\end{proof} \\ \\
 To derive a more explicit condition for fixation, we let
 $$ H_N \ := \ \{T (z) < \infty \ \hbox{for some} \ z < - N \} $$
 be the event introduced in~\eqref{eq:fixation}.
 Following the construction in~\cite{bramson_griffeath_1989}, we also let~$\tau_N$ be the first time an active path starting from the
 left of~$- N$ hits the origin, and observe that
 $$ \tau_N \ = \ \inf \,\{T (z) : z \in (- \infty, - N) \}. $$
 In particular, the event~$H_N$ can be written as
\begin{equation}
\label{eq:key-event}
  \begin{array}{l} H_N \ = \ \bigcap_{z < - N} \, \{T (z) < \infty \} \ = \ \{\tau_N < \infty \}. \end{array}
\end{equation}
 Given the event~$H_N$, we let~$z_- < - N$ and~$z_+ \geq 0$ be the initial position of the active path and the rightmost
 source of an active path that reaches the origin by time~$\tau_N$, i.e.,
\begin{equation}
\label{eq:paths}
  \begin{array}{rcl}
    z_- & := & \,\min \,\{z \in \Z : (z, 0) \leadsto (0, \tau_N) \} \ < \ - N \vspace*{2pt} \\
    z_+ & := & \max \,\{z \in \Z : (z, 0) \leadsto (0, \sigma_N) \ \hbox{for some} \ \sigma_N < \tau_N \} \ \geq \ 0, \end{array}
\end{equation}
 and define~$I_N = (z_-, z_+)$.
 Now, we observe that, on the event~$H_N$,
\begin{itemize}
\item All the frozen piles initially in~$I_N$ must have been destroyed, i.e., turned into active piles due to the occurrence of
 annihilating events, by time~$\tau_N$. \vspace*{4pt}
\item The active particles initially outside the interval~$I_N$ cannot jump inside the space-time region delimited
 by the two active paths implicitly defined in~\eqref{eq:paths} because the existence of such particles would contradict
 the minimality of~$z_-$ or the maximality of~$z_+$.
\end{itemize}
 This, together with equation~\eqref{eq:key-event}, implies that, given the event~$H_N$, all the frozen piles initially in the random
 interval~$I_N$ must have been destroyed by either active piles initially in this interval or active piles that result from
 the destruction of these frozen piles.
 To quantify this statement, we attach random variables, that we call {\bf contributions}, to each edge.
 The definition depends on whether the edge initially carries an active pile or a frozen pile.
 To begin with, we give an arbitrary deterministic contribution, say~$-1$, to each pile initially active by setting
\begin{equation}
\label{eq:contribution-active}
 \cont (e) \ := \ - 1 \quad \hbox{whenever} \quad 0 < \xi_0 (e) \leq \threshold.
\end{equation}
 Now, we observe that, given~$H_N$, for each frozen pile initially in~$I_N$, a random number of active piles must have jumped onto this frozen
 pile to turn it into an active pile.
 Therefore, to define the contribution of a frozen pile, we let
\begin{equation}
\label{eq:breaks}
  T_e \ := \ \inf \,\{t > 0 : \xi_t (e) \leq \threshold \}
\end{equation}
 and define the contribution of a frozen pile initially at~$e$ as
\begin{equation}
\label{eq:contribution-frozen}
 \cont (e) \ := \ - 1 + \hbox{number of active piles that hit~$e$ until time~$T_e$}.
\end{equation}
 Note that~\eqref{eq:contribution-frozen} reduces to~\eqref{eq:contribution-active} when edge~$e$ carries initially an active pile since in this
 case the time until the edge becomes active is zero, therefore~\eqref{eq:contribution-frozen} can be used as the general definition for the contribution
 of an edge with at least one particle.
 Edges with initially no particle have contribution zero.
 Since the occurrence of~$H_N$ implies that all the frozen piles initially in~$I_N$ must have been destroyed by either active piles initially
 in this interval or active piles that result from the destruction of these frozen piles, in which case~$T_e < \infty$ for all the
 edges in the interval, and since particles in an active pile jump all at once rather than individually,
\begin{equation}
\label{eq:inclusion-1}
 \begin{array}{rcl}
    H_N & \subset & \{\sum_{e \in I_N} \,\cont (e \,| \,T_e < \infty) \leq 0 \} \vspace*{4pt} \\
        & \subset & \{\sum_{e \in (l, r)} \,\cont (e \,| \,T_e < \infty) \leq 0 \ \hbox{for some~$l < - N$ and some~$r \geq 0$} \}. \end{array}
\end{equation}
 Lemma~\ref{lem:fixation-condition} and~\eqref{eq:inclusion-1} are used in section~\ref{sec:fixation} together with the large
 deviation estimates for the number of active and frozen piles showed in the following section to prove Theorem~\ref{th:fixation}.

%%%%%%%%%%%%%%%%%%%%%%%%%%%%%%%%%%%%%%%%%%%%%%%%%%%%%%%%%%%%%%%%%%%%%%%%%%%%%%%%%%%%%%%%%%%%%%%%%%%%%%%%%%%%%%%%%%%%%%%%%%%%%%%%%%%%%%%%%%

\section{Large deviation estimates}
\label{sec:deviation}

\indent In order to find later a good upper bound for the probability in~\eqref{eq:fixation} and deduce a sufficient condition for
 fixation of the opinion model, the next step is to prove large deviation estimates for the initial number of piles with~$s$~particles
 in a large interval.
 More precisely, the main objective of this section is to prove that for all~$s$ and all~$\ep > 0$ the probability that
 $$ \card \{e \in (0, N) : \xi_0 (e) = s \} \ \notin \ (1 - \ep, 1 + \ep) \ E \,(\card \{e \in (0, N) : \xi_0 (e) = s \}) $$
 decays exponentially with~$N$.
 Note that, even though it is assumed that the process starts from a product measure and therefore the initial opinions at different
 vertices are chosen independently, the initial states at different edges are not independent in general.
 When starting from the uniform product measure, these states are independent if and only if, for every size~$s$,
 $$ \card \{j \in \VO : d (i, j) = s \} \quad \hbox{does not depend on~$i \in \VO$}. $$
 This holds for cycles or hypercubes but not for graphs which are not vertex-transitive.
 When starting from more general product measures, the initial number of particles at different edges are not independent, even for
 very specific graphs.
 In particular, the number of piles of particles with a given size in a given interval does not simply reduce to a binomial random variable.

\indent The main ingredient to prove large deviation estimates for the initial number of piles with a given number of particles in
 a large spatial interval is to first show large deviation estimates for the number of so-called changeovers in a sequence of independent
 coin flips.
 Consider an infinite sequence of independent coin flips such that
 $$ P \,(X_t = H) = p \quad \hbox{and} \quad P \,(X_t = T) = q = 1 - p \quad \hbox{for all} \quad t \in \N $$
 where~$X_t$ is the outcome: heads or tails, at time~$t$.
 We say that a {\bf changeover} occurs whenever two consecutive coin flips result in two different outcomes.
 The expected value of the number of changeovers~$Z_N$ before time~$N$ can be easily computed by observing that
 $$ \begin{array}{l} Z_N \ = \ \sum_{0 \leq t < N} \,Y_t \quad \hbox{where} \quad Y_t \ := \ \ind \{X_{t + 1} \neq X_t \} \end{array} $$
 and by using the linearity of the expected value:
 $$ \begin{array}{rcl}
      E Z_N & = & \sum_{0 \leq t < N} \,E Y_t \vspace*{4pt} \\
            & = & \sum_{0 \leq t < N} \,P \,(X_{t + 1} \neq X_t) \ = \ N \,P \,(X_0 \neq X_1) \ = \ 2 N p q. \end{array} $$
 Then, we have the following result for the number of changeovers.
\begin{lemma} --
\label{lem:changeover}
 For all~$\ep > 0$, there exists~$c_0 > 0$ such that
 $$ P \,(Z_N - E Z_N \notin (- \ep N, \ep N)) \ \leq \ \exp (- c_0 N) \quad \hbox{for all~$N$ large}. $$
\end{lemma}
\begin{proof}
 To begin with, we let~$\tau_{2K}$ be the time to the~$2K$th changeover and notice that, since all the outcomes between two consecutive
 changeovers are identical, the sequence of coin flips up to this stopping time can be decomposed into~$2K$ strings with an alternation
 of strings with only heads and strings with only tails followed by one more coin flip.
 In addition, since the coin flips are independent, the length distribution of each string is
 $$ \begin{array}{rcl}
      H_j & := & \hbox{length of the~$j$th string of heads} \ = \ \geometric (q) \vspace*{2pt} \\
      T_j & := & \hbox{length of the~$j$th string of tails} \ = \ \geometric (p) \end{array} $$
 and lengths are independent.
 In particular,~$\tau_{2K}$ is equal in distribution to the sum of~$2K$ independent geometric random variables with parameters~$p$ and~$q$,
 therefore we have
\begin{equation}
\label{eq:changeover-1}
  P \,(\tau_{2K} = n) \ = \ P \,(H_1 + T_1 + \cdots + H_K + T_K = n) \quad \hbox{for all} \quad n \in \N.
\end{equation}
 Now, observing that, for all~$K \leq n$,
 $$ \begin{array}{l}
    \displaystyle P \,(H_1 + H_2 + \cdots + H_K = n) \ = \,{n - 1 \choose K - 1} \,q^K \,(1 - q)^{n - K} \vspace*{2pt} \\ \hspace*{50pt}
    \displaystyle = \,\frac{K}{n} \ {n \choose K} \,q^K \,(1 - q)^{n - K} \ \leq \ P \,(\binomial (n, q) = K), \end{array} $$
 large deviation estimates for the binomial distribution imply that
\begin{equation}
\label{eq:changeover-2}
  \begin{array}{l}
    P \,((1/K)(H_1 + H_2 + \cdots + H_K) \geq (1 + \ep)(1/q)) \vspace*{4pt} \\ \hspace*{70pt} \leq \
    P \,(\binomial ((1 + \ep)(1/q) K, q) \leq K) \ \leq \ \exp (- c_1 K) \vspace*{8pt} \\
    P \,((1/K)(H_1 + H_2 + \cdots + H_K) \leq (1 - \ep)(1/q)) \vspace*{4pt} \\ \hspace*{70pt} \leq \
    P \,(\binomial ((1 - \ep)(1/q) K, q) \geq K) \ \leq \ \exp (- c_1 K) \end{array}
\end{equation}
 for a suitable constant~$c_1 > 0$ and all~$K$ large.
 Similarly, for all~$\ep > 0$,
\begin{equation}
\label{eq:changeover-3}
  \begin{array}{l}
    P \,((1/K)(T_1 + T_2 + \cdots + T_K) \geq (1 + \ep)(1/p)) \ \leq \ \exp (- c_2 K) \vspace*{4pt} \\
    P \,((1/K)(T_1 + T_2 + \cdots + T_K) \leq (1 - \ep)(1/p)) \ \leq \ \exp (- c_2 K) \end{array}
\end{equation}
 for a suitable~$c_2 > 0$ and all~$K$ large.
 Combining~\eqref{eq:changeover-1}--\eqref{eq:changeover-3}, we deduce that
 $$ \begin{array}{l}
     P \,((1/K) \,\tau_{2K} \notin ((1 - \ep)(1/p + 1/q), (1 + \ep)(1/p + 1/q))) \vspace*{4pt} \\ \hspace*{20pt} = \
     P \,((1/K)(H_1 + T_1 + \cdots + H_K + T_K) \notin ((1 - \ep)(1/p + 1/q), (1 + \ep)(1/p + 1/q))) \vspace*{4pt} \\ \hspace*{20pt} \leq \
     P \,((1/K)(H_1 + H_2 + \cdots + H_K) \notin ((1 - \ep)(1/q), (1 + \ep)(1/q))) \vspace*{4pt} \\ \hspace*{80pt} + \
     P \,((1/K)(T_1 + T_2 + \cdots + T_K) \notin ((1 - \ep)(1/p), (1 + \ep)(1/p))) \vspace*{4pt} \\ \hspace*{20pt} \leq \
     2 \exp (- c_1 K) + 2 \exp (- c_2 K). \end{array} $$
 Taking~$K := pq N$ and using that~$pq \,(1/p + 1/q) = 1$, we deduce
 $$ \begin{array}{l}
      P \,((1/N) \,\tau_{2K} \notin (1 - \ep, 1 + \ep)) \vspace*{4pt} \\ \hspace*{20pt} = \
      P \,((1/K) \,\tau_{2K} \notin ((1 - \ep) (1/p + 1/q), (1 + \ep)(1/p + 1/q))) \ \leq \ \exp (- c_3 N) \end{array} $$
 for a suitable~$c_3 > 0$ and all~$N$ large.
 In particular,
 $$ \begin{array}{rcl}
      P \,((1/N) \,\tau_{2K - \ep N} \geq 1) & \leq & \exp (- c_4 N) \vspace*{4pt} \\
      P \,((1/N) \,\tau_{2K + \ep N} \leq 1) & \leq & \exp (- c_5 N) \end{array} $$
 for suitable constants~$c_4 > 0$ and~$c_5 > 0$ and all~$N$ sufficiently large.
 Using the previous two inequalities and the fact that the event that the number of changeovers before time~$N$ is equal to~$2K$
 is also the event that the time to the~$2K$th changeover is less than~$N$ but the time to the next changeover is more than~$N$,
 we conclude that
 $$ \begin{array}{l}
      P \,(Z_N - E Z_N \notin (- \ep N, \ep N)) \ = \
      P \,(Z_N \notin (2 pq - \ep, 2 pq + \ep) N) \vspace*{4pt} \\ \hspace*{20pt} = \
      P \,((1/N) \,Z_N \notin (2 pq - \ep, 2 pq + \ep)) \vspace*{4pt} \\ \hspace*{20pt} = \
      P \,((1/N) \,Z_N \leq 2 pq - \ep) + P \,((1/N) \,Z_N \geq 2 pq + \ep) \vspace*{4pt} \\ \hspace*{20pt} = \
      P \,((1/N) \,\tau_{2K - \ep N} \geq 1) + P \,((1/N) \,\tau_{2K + \ep N} \leq 1) \ \leq \
        \exp (- c_4 N) + \exp (- c_5 N) \end{array} $$
 for all~$N$ large.
 This completes the proof.
\end{proof} \\ \\
 Now, we say that an edge is of type~$i \to j$ if it connects an individual with initial opinion~$i$ on the left to an individual
 with initial opinion~$j$ on the right, and let
 $$ e_N (i \to j) \ := \ \card \{x \in (0, N) : \eta_0 (x) = i \ \hbox{and} \ \eta_0 (x + 1) = j \} $$
 denote the number of edges of type~$i \to j$ in the interval~$(0, N)$.
 Using the large deviation estimates for the number of changeovers established in the previous lemma, we can deduce large
 deviation estimates for the number of edges of each type.
\begin{lemma} --
\label{lem:edge}
 For all~$\ep > 0$, there exists~$c_6 > 0$ such that
 $$ P \,(e_N (i \to j) - N \rho_i \,\rho_j \notin (- 2 \ep N, 2 \ep N)) \ \leq \ \exp (- c_6 N) \quad \hbox{for all~$N$ large and~$i \neq j$}. $$
\end{lemma}
\begin{proof}
 For any given~$i$, the number of edges~$i \to j$ and~$j \to i$  with~$j \neq i$ has the same distribution as the number of changeovers in a
 sequence of independent coin flips of a coin that lands on heads with probability~$\rho_i$.
 In particular, applying Lemma~\ref{lem:changeover} with~$p = \rho_i$ gives
\begin{equation}
\label{eq:edge-1}
  \begin{array}{l} P \,(\sum_{j \neq i} \,e_N (i \to j) - N \rho_i \,(1 - \rho_i) \notin (- \ep N, \ep N)) \ \leq \ \exp (- c_0 N) \end{array}
\end{equation}
 for all~$N$ sufficiently large.
 In addition, since each~$i$ preceding a changeover is independently followed by any of the remaining~$F - 1$ opinions, for all~$i \neq j$, we have
\begin{equation}
\label{eq:edge-2}
  \begin{array}{l}
    P \,(e_N (i \to j) = n \ | \,\sum_{k \neq i} \,e_N (i \to k) = K) \vspace*{4pt} \\ \hspace*{50pt} = \
    P \,(\binomial (K, \rho_j \,(1 - \rho_i)^{-1}) = n). \end{array}
\end{equation}
 Combining~\eqref{eq:edge-1}--\eqref{eq:edge-2} with large deviation estimates for the binomial distribution, conditioning on the number
 of edges of type~$i \to k$ for some~$k \neq i$, and using that
 $$ (N \rho_i \,(1 - \rho_i) + \ep N) \,\rho_j \,(1 - \rho_i)^{-1} \ = \ N \rho_i \,\rho_j + \ep N \rho_j \,(1 - \rho_i)^{-1} $$
 we deduce the existence of~$c_7 > 0$ such that
\begin{equation}
\label{eq:edge-3}
  \begin{array}{l}
    P \,(e_N (i \to j) - N \rho_i \,\rho_j \geq 2 \ep N) \vspace*{4pt} \\ \hspace*{20pt} \leq \
    P \,(\sum_{k \neq i} \,e_N (i \to k) - N \rho_i \,(1 - \rho_i) \geq \ep N) \vspace*{4pt} \\ \hspace*{20pt} + \
    P \,(e_N (i \to j) \geq N \rho_i \,\rho_j + 2 \ep N \ | \ \sum_{k \neq i} \,e_N (i \to k) - N \rho_i \,(1 - \rho_i) < \ep N) \vspace*{4pt} \\ \hspace*{20pt} \leq \
      \exp (- c_0 N) + P \,(\binomial (N \rho_i \,(1 - \rho_i) + \ep N, \rho_j \,(1 - \rho_i)^{-1}) \geq N \rho_i \,\rho_j + 2 \ep N) \vspace*{4pt} \\ \hspace*{20pt} \leq \
      \exp (- c_0 N) + \exp (- c_7 N) \end{array}
\end{equation}
 for all~$N$ large.
 Similarly, there exists~$c_8 > 0$ such that
\begin{equation}
\label{eq:edge-4}
  P \,(e_N (i \to j) - N \rho_i \,\rho_j \leq - 2 \ep N) \ \leq \ \exp (- c_0 N) + \exp (- c_8 N)
\end{equation}
 for all~$N$ large.
 The lemma follows from~\eqref{eq:edge-3}--\eqref{eq:edge-4}.
\end{proof} \\ \\
 Note that the large deviation estimates for the initial number of piles of particles easily follows from the previous lemma.
 Finally, from the large deviation estimates for the number of edges of each type, we deduce the analog for a general class of
 functions~$\weight$ that will be used in the next section to prove the first sufficient condition for fixation.
\begin{lemma} --
\label{lem:weight}
 Let~$w : \VO \times \VO \to \R$ be any function such that
 $$ w (i, i) = 0 \quad \hbox{for all} \quad i \in \VO $$
 and let~$\weight : \Z + 1/2 \to \R$ be the function defined as
 $$ \weight_e := w (i, j) \quad \hbox{whenever} \quad \hbox{edge~$e$ is of type~$i \to j$}. $$
 Then, for all~$\ep > 0$, there exists~$c_9 > 0$ such that
 $$ \begin{array}{l} P \,(\sum_{e \in (0, N)} \,(\weight_e - E \weight_e) \notin (- \ep N, \ep N)) \ \leq \ \exp (- c_9 N) \quad \hbox{for all~$N$ large}. \end{array} $$
\end{lemma}
\begin{proof}
 First, we observe that
 $$ \begin{array}{l}
     \sum_{e \in (0, N)} \,(\weight_e - E \weight_e) \ = \
     \sum_{e \in (0, N)} \weight_e - N E \weight_e \vspace*{4pt} \\ \hspace*{20pt} = \
     \sum_{i \neq j} \,w (i, j) \,e_N (i \to j) - N \,\sum_{i \neq j} \,w (i, j) \,P \,(e \ \hbox{is of type} \ i \to j) \vspace*{4pt} \\ \hspace*{20pt} = \
     \sum_{i \neq j} \,w (i, j) \,(e_N (i \to j) - N \rho_i \,\rho_j). \end{array} $$
 Letting~$m := \max_{i \neq j} |w (i, j)| < \infty$ and applying Lemma~\ref{lem:edge}, we conclude that
 $$ \begin{array}{l}
      P \,(\sum_{e \in (0, N)} \,(\weight_e - E \weight_e) \notin (- \ep N, \ep N)) \vspace*{4pt} \\ \hspace*{25pt} = \
      P \,(\sum_{i \neq j} \,w (i, j) \,(e_N (i \to j) - N \rho_i \,\rho_j) \notin (- \ep N, \ep N)) \vspace*{4pt} \\ \hspace*{25pt} \leq \
      P \,(w (i, j) \,(e_N (i \to j) - N \rho_i \,\rho_j) \notin (- \ep N/F^2, \ep N/F^2) \ \hbox{for some} \ i \neq j) \vspace*{4pt} \\ \hspace*{25pt} \leq \
      P \,(e_N (i \to j) - N \rho_i \,\rho_j \notin (- \ep N/mF^2, \ep N/mF^2) \ \hbox{for some} \ i \neq j) \vspace*{4pt} \\ \hspace*{25pt} \leq \
      F^2 \,\exp (- c_{10} N) \end{array} $$
 for a suitable constant~$c_{10} > 0$ and all~$N$ large.
\end{proof}

%%%%%%%%%%%%%%%%%%%%%%%%%%%%%%%%%%%%%%%%%%%%%%%%%%%%%%%%%%%%%%%%%%%%%%%%%%%%%%%%%%%%%%%%%%%%%%%%%%%%%%%%%%%%%%%%%%%%%%%%%%%%%%%%%%%%%%%%%%

\section{Proof of Theorem~\ref{th:fixation} (general opinion graphs)}
\label{sec:fixation}

\indent The key ingredients to prove Theorem~\ref{th:fixation} are Lemma~\ref{lem:fixation-condition} and inclusions~\eqref{eq:inclusion-1}.
 The large deviation estimates of the previous section are also important to make the sufficient condition for fixation more explicit and
 applicable to particular opinion graphs.
 First, we find a lower bound~$\weight_e$, that we shall call {\bf weight}, for the contribution of any given edge~$e$.
 This lower bound is deterministic given the initial number of particles at the edge and is obtained assuming the worst case scenario
 where all the active piles annihilate with frozen piles rather than other active piles.
 More precisely, we have the following lemma.
\begin{lemma} --
\label{lem:deterministic}
 For all~$k > 0$,
\begin{equation}
\label{eq:weight}
  \cont (e \,| \,T_e < \infty) \ \geq \ \weight_e \ := \ k - 2 \quad \hbox{when} \quad (k - 1) \,\threshold < \xi_0 (e) \leq k \threshold. 
\end{equation}
\end{lemma}
\begin{proof}
 The jump of an active pile of size~$s_- \leq \threshold$ onto a frozen pile of size~$s_+ > \threshold$ decreases the size of this frozen
 pile by at most~$s_-$ particles.
 Since in addition active piles have at most~$\threshold$ particles, whenever the initial number of frozen particles at edge~$e$ satisfies
 $$ (k - 1) \,\threshold < \xi_0 (e) \leq k \threshold \quad \hbox{for some} \quad k \geq 2, $$
 at least~$k - 1$ active piles must have jumped onto~$e$ until time~$T_e < \infty$.
 Recalling~\eqref{eq:contribution-frozen} gives the result when edge~$e$ carries a frozen while the result is trivial when the
 edge carries an active pile since, in this case, both its contribution and its weight are equal to~$-1$.
\end{proof} \\ \\
 In view of Lemma~\ref{lem:deterministic}, it is convenient to classify piles depending on the number of complete blocks of~$\threshold$ particles
 they contain: we say that the pile at~$e$ is of {\bf order}~$k > 0$ when
 $$ (k - 1) \,\threshold < \xi_t (e) \leq k \threshold \quad \hbox{or equivalently} \quad \ceil{\xi_t (e) / \threshold} = k $$
 so that active piles are exactly the piles of order one and the weight of a pile is simply its order minus two.
 Now, we note that Lemma~\ref{lem:deterministic} and~\eqref{eq:inclusion-1} imply that
\begin{equation}
\label{eq:inclusion-2}
 \begin{array}{l}
    H_N \ \subset \ \{\sum_{e \in (l, r)} \weight_e \leq 0 \ \hbox{for some~$l < - N$ and some~$r \geq 0$} \}. \end{array}
\end{equation}
 Motivated by Lemma~\ref{lem:fixation-condition}, the main objective to study fixation is to find an upper bound for the probability of the
 event on the right-hand side of~\eqref{eq:inclusion-2}.
 This is the key to proving the following general fixation result from which both parts of Theorem~\ref{th:fixation} can be easily deduced.
\begin{lemma} --
\label{lem:expected-weight}
 Assume~\eqref{eq:product}.
 Then, the system on~$\Z$ fixates whenever
 $$ \begin{array}{l} \sum_{i, j \in \VO} \,\rho_i \,\rho_j \,\sum_{k > 0} \,((k - 2) \,\sum_{s : \ceil{s / \threshold} = k} \,\ind \{d (i, j) = s \}) \ > \ 0. \end{array} $$
\end{lemma}
\begin{proof}
 To begin with, we observe that
 $$ \begin{array}{rcl}
      P \,(\xi_0 (e) = s) & = & \sum_{i, j \in \VO} \,P \,(\eta_0 (x) = i \ \hbox{and} \ \eta_0 (x + 1) = j) \ \ind \{d (i, j) = s \} \vspace*{4pt} \\
                          & = & \sum_{i, j \in \VO} \,\rho_i \,\rho_j \ \ind \{d (i, j) = s \}. \end{array} $$
 Recalling~\eqref{eq:weight}, it follows that
 $$ \begin{array}{rcl}
      E \weight_e & = & \sum_{k > 0} \,(k - 2) \,P \,((k - 1) \,\threshold < \xi_0 (e) \leq k \threshold) \vspace*{4pt} \\
                 & = & \sum_{k > 0} \,((k - 2) \,\sum_{(k - 1) \,\threshold < s \leq k \threshold} \,P \,(\xi_0 (e) = s)) \vspace*{4pt} \\
                 & = & \sum_{k > 0} \,((k - 2) \,\sum_{(k - 1) \,\threshold < s \leq k \threshold} \,\sum_{i, j \in \VO} \,\rho_i \,\rho_j \ \ind \{d (i, j) = s \}) \vspace*{4pt} \\
                 & = & \sum_{i, j \in \VO} \,\rho_i \,\rho_j \,\sum_{k > 0} \,((k - 2) \,\sum_{s : \ceil{s / \threshold} = k} \,\ind \{d (i, j) = s \}) \end{array} $$
 which is strictly positive under the assumption of the lemma.
 In particular, applying the large deviation estimate in Lemma~\ref{lem:weight} with~$\ep := E \weight_e > 0$, we deduce that
 $$ \begin{array}{l}
      P \,(\sum_{e \in (0, N)} \weight_e \leq 0) \ = \
      P \,(\sum_{e \in (0, N)} \,(\weight_e - E \weight_e) \leq - \ep N) \vspace*{4pt} \\ \hspace*{40pt} \leq \
      P \,(\sum_{e \in (0, N)} \,(\weight_e - E \weight_e) \notin (- \ep N, \ep N)) \ \leq \ \exp (- c_9 N) \end{array} $$
 for all~$N$ large, which, in turn, implies with~\eqref{eq:inclusion-2} that
 $$ \begin{array}{rcl}
      P \,(H_N) & \leq &
      P \,(\sum_{e \in (l, r)} \weight_e \leq 0 \ \hbox{for some~$l < - N$ and~$r \geq 0$}) \vspace*{4pt} \\ & \leq &
           \sum_{l < - N} \,\sum_{r \geq 0} \,\exp (- c_9 \,(r - l)) \ \to \ 0 \end{array} $$
 as~$N \to \infty$.
 This together with Lemma~\ref{lem:fixation-condition} implies fixation.
\end{proof} \\ \\
 Both parts of Theorem~\ref{th:fixation} directly follow from the previous lemma. \\ \\
\begin{demo}{Theorem~\ref{th:fixation}a} --
 Assume that~\eqref{eq:uniform} holds and that
 $$ \begin{array}{l} S (\OO, \threshold) \ = \ \sum_{k > 0} \,((k - 2) \,\sum_{s : \ceil{s / \threshold} = k} \,N (\OO, s)) \ > \ 0. \end{array} $$
 Then, the expected weight becomes
 $$ \begin{array}{rcl}
      E \weight_e & = & \sum_{i, j \in \VO} \,\rho_i \,\rho_j \,\sum_{k > 0} \,((k - 2) \,\sum_{s : \ceil{s / \threshold} = k} \,\ind \{d (i, j) = s \}) \vspace*{4pt} \\
                 & = & (1/F)^2 \,\sum_{k > 0} \,((k - 2) \,\sum_{s : \ceil{s / \threshold} = k} \,\sum_{i, j \in \VO} \,\ind \{d (i, j) = s \}) \vspace*{4pt} \\
                 & = & (1/F)^2 \,\sum_{k > 0} \,((k - 2) \,\sum_{s : \ceil{s / \threshold} = k} \,\card \{(i, j) \in \VO \times \VO : d (i, j) = s \}) \vspace*{4pt} \\
                 & = & (1/F)^2 \,\sum_{k > 0} \,((k - 2) \,\sum_{s : \ceil{s / \threshold} = k} \,N (\OO, s)) \vspace*{4pt} \\
                 & = & (1/F)^2 \,S (\OO, \threshold) \ > \ 0 \end{array} $$
 which, according to Lemma~\ref{lem:expected-weight}, implies fixation.
\end{demo} \\ \\
\begin{demo}{Theorem~\ref{th:fixation}b} --
 Assume that~$\diameter > 2 \threshold$. Then,
 $$ d (i_-, i_+) = \diameter > 2 \threshold \quad \hbox{for some pair} \quad (i_-, i_+) \in \VO \times \VO. $$
 Now, let~$X, Y \geq 0$ such that~$2X + (F - 2)Y = 1$ and assume that
 $$ \rho_{i_-} = \rho_{i_+} = X \quad \hbox{and} \quad \rho_i = Y \quad \hbox{for all} \quad i \notin B := \{i_-, i_+ \}. $$
 To simplify the notation, we also introduce
 $$ \begin{array}{l} Q (i, j) \ := \ \sum_{k > 0} \,((k - 2) \,\sum_{s : \ceil{s / \threshold} = k} \,\ind \{d (i, j) = s \}) \end{array} $$
 for all~$(i, j) \in \VO \times \VO$.
 Then, the expected weight becomes
 $$ \begin{array}{rcl}
      P (X, Y) & = & \sum_{i, j \in \VO} \,\rho_i \,\rho_j \ Q (i, j) \vspace*{4pt} \\
               & = & \sum_{i, j \in B} \,\rho_i \,\rho_j \ Q (i, j) + \sum_{i \notin B} \,2 \,\rho_i \,\rho_{i_-} \,Q (i, i_-) \vspace*{4pt} \\ && \hspace*{25pt} + \
                     \sum_{i \notin B} \,2 \,\rho_i \,\rho_{i_+} \,Q (i, i_+) + \sum_{i, j \notin B} \,\rho_i \,\rho_j \ Q (i, j) \vspace*{4pt} \\
               & = & 2 \,Q (i_-, i_+) \,X^2 + 2 \,(\sum_{i \notin B} \,Q (i, i_-) + Q (i, i_+)) \,XY + \sum_{i, j \notin B} \,Q (i, j) \,Y^2. \end{array} $$
 This shows that~$P$ is continuous in both~$X$ and~$Y$ and that
 $$ \begin{array}{rcl}
      P (1/2, 0) & = & (1/2) \,Q (i_-, i_+) \vspace*{4pt} \\
                 & = & (1/2) \,\sum_{k > 0} \,((k - 2) \,\sum_{s : \ceil{s / \threshold} = k} \,\ind \{d (i_-, i_+) = s \}) \vspace*{4pt} \\
                 & \geq & (1/2) \,(3 - 2) \,\sum_{s > 2 \threshold} \,\ind \{d (i_-, i_+) = s \} \ = \ 1/2 \ > \ 0. \end{array} $$
 Therefore, according to Lemma~\ref{lem:expected-weight}, there is fixation of the one-dimensional process starting from any product
 measure whose densities are in some neighborhood of
 $$ \rho_{i_-} = \rho_{i_+} = 1/2 \quad \hbox{and} \quad \rho_i = 0 \quad \hbox{for all} \quad i \notin \{i_-, i_+ \}. $$
 This proves the second part of Theorem~\ref{th:fixation}.
\end{demo}

%%%%%%%%%%%%%%%%%%%%%%%%%%%%%%%%%%%%%%%%%%%%%%%%%%%%%%%%%%%%%%%%%%%%%%%%%%%%%%%%%%%%%%%%%%%%%%%%%%%%%%%%%%%%%%%%%%%%%%%%%%%%%%%%%%%%%%%%%%

\section{Proof of Theorem~\ref{th:dist-reg} (distance-regular graphs)}
\label{sec:dist-reg}

\indent To explain the intuition behind the proof, recall that, when an active pile of size~$s_-$ jumps to the right onto a frozen
 pile of size~$s_+$ at edge~$e$, the size of the latter pile becomes
 $$ \xi_t (e) \ = \ d (\eta_t (e - 1/2), \eta_t (e + 1/2)) \ = \ d (\eta_{t-} (e - 3/2), \eta_{t-} (e + 1/2)) $$
 and the triangle inequality implies that
\begin{equation}
\label{eq:triangle}
  s_+ - s_- \ = \ \xi_{t-} (e) - \xi_{t-} (e - 1) \ \leq \ \xi_t (e) \ \leq \ \xi_{t-} (e) + \xi_{t-} (e - 1) \ = \ s_+ + s_-.
\end{equation}
 The exact distribution of the new size cannot be deduced in general from the size of the intersecting piles, indicating that
 the system of piles is not Markov.
 The key to the proof is that, at least when the underlying opinion graph is distance-regular, the system of piles becomes Markov.
 The first step is to show that, for all opinion graphs, the opinions on the left and on the right of a pile of size~$s$ are
 conditioned to be at distance~$s$ of each other but are otherwise independent, which follows from the fact that both opinions
 originate from two different ancestors at time zero, and the fact that the initial distribution is a product measure.
 If in addition the opinion graph is distance-regular then the number of possible opinions on the left and on the right of the pile,
 which is also the number of pairs of opinions at distance~$s$ of each other, does not depend on the actual opinion on the left of
 the pile.
 This implies that, at least in theory, the new size distribution of a pile right after a collision can be computed explicitly.
 This is then used to prove that a jump of an active pile onto a pile of order~$n > 1$ reduces its order with probability at most
 $$ \begin{array}{l} p_n \ = \ \max \,\{\sum_{s : \ceil{s / \threshold} = n - 1} f (s_-, s_+, s) / h (s_+) : \ceil{s_- / \threshold} = 1 \ \hbox{and} \ \ceil{s_+ / \threshold} = n \} \end{array} $$
 while it increases its order with probability at least
 $$ \begin{array}{l} q_n \ = \ \,\min \,\{\sum_{s : \ceil{s / \threshold} = n + 1} f (s_-, s_+, s) / h (s_+) : \ceil{s_- / \threshold} = 1 \ \hbox{and} \ \ceil{s_+ / \threshold} = n \}. \end{array} $$
 In particular, the number of active piles that need to be sacrificed to turn a frozen pile into an active pile is stochastically
 larger than the hitting time to state~1 of a certain discrete-time birth and death process.
 To turn this into a proof, we let~$x = e - 1/2$ and
 $$ x - 1 \to_t x \ := \ \hbox{the event that there is an arrow~$x - 1 \to x$ at time~$t$}. $$
 Then, we have the following lemma.
\begin{lemma} --
\label{lem:collision}
 Assume~\eqref{eq:uniform} and~\eqref{eq:dist-reg-1}.
 For all~$s \geq 0$ and $s_-, s_+ > 0$ with~$s_- \leq \threshold$,
 $$ \begin{array}{l} P \,(\xi_t (e) = s \,| \,(\xi_{t-} (e - 1), \xi_{t-} (e)) = (s_-, s_+) \ \hbox{and} \ x - 1 \to_t x) \ = \ f (s_-, s_+, s) / h (s_+). \end{array} $$
\end{lemma}
\begin{proof}
 The first step is similar to the proof of~\cite[Lemma~3]{lanchier_scarlatos_2013}.
 Due to one-dimensional nearest neighbor interactions, active paths cannot cross each other.
 In particular, the opinion dynamics preserve the ordering of the ancestral lineages therefore
\begin{equation}
\label{eq:collision-1}
  a (x - 1, t-) \ \leq \ a (x, t-) \ \leq \ a (x + 1, t-)
\end{equation}
 where~$a (z, t-)$ refers to the ancestor of~$(z, t-)$, i.e., the unique source at time zero of an active path reaching space-time point~$(z, t-)$.
 Since in addition~$s_-, s_+ > 0$, given the conditioning in the statement of the lemma, the individuals at~$x$ and~$x \pm 1$ must disagree at
 time~$t-$ and so must have different ancestors.
 This together with~\eqref{eq:collision-1} implies that
\begin{equation}
\label{eq:collision-2}
  a (x - 1, t-) \ < \ a (x, t-) \ < \ a (x + 1, t-).
\end{equation}
 Now, we fix~$i_-, j \in \VO$ such that~$d (i_-, j) = s_-$ and let
 $$ B_{t-} (i_-, j) \ := \ \{\eta_{t-} (x - 1) = i_- \ \hbox{and} \ \eta_{t-} (x) = j \}. $$
 Then, given this event and the conditioning in the statement of the lemma, the probability that the pile of particles at~$e$ becomes of size~$s$
 is equal to
\begin{equation}
\label{eq:collision-3}
  \begin{array}{l}
    P \,(\xi_t (e) = s \,| \,B_{t-} (i_-, j) \ \hbox{and} \ \xi_{t-} (e) = s_+ \ \hbox{and} \ x - 1 \to_t x) \vspace*{4pt} \\ \hspace*{25pt} = \
    P \,(d (i_-, \eta_{t-} (x + 1)) = s \,| \,B (i_-, j) \ \hbox{and} \vspace*{4pt} \\ \hspace*{100pt}
         d (j, \eta_{t-} (x + 1)) = s_+ \ \hbox{and} \ x - 1 \to_t x) \vspace*{4pt} \\ \hspace*{25pt} = \
      \card \{i_+ : d (i_-, i_+) = s \ \hbox{and} \ d (i_+, j) = s_+ \} / \card \{i_+ : d (i_+, j) = s_+ \} \end{array}
\end{equation}
 where the last equality follows from~\eqref{eq:uniform} and~\eqref{eq:collision-2} which, together, imply that the opinion at~$x + 1$ just before
 the jump is independent of the other opinions on its left and chosen uniformly at random from the set of opinions at distance~$s_+$ of opinion~$j$.
 Assuming in addition that the underlying opinion graph is distance-regular~\eqref{eq:dist-reg-1}, we also have
\begin{equation}
\label{eq:collision-4}
  \begin{array}{l}
    \card \{i_+ : d (i_-, i_+) = s \ \hbox{and} \ d (i_+, j) = s_+ \} \vspace*{4pt} \\ \hspace*{50pt} = \
      N (\OO, (i_-, s), (j, s_+)) \ = \ f (s_-, s_+, s) \vspace*{8pt} \\
    \card \{i_+ : d (i_+, j) = s_+ \} \ = \ N (\OO, (j, s_+)) \ = \ h (s_+). \end{array}
\end{equation}
 In particular, the conditional probability in~\eqref{eq:collision-3} does not depend on the particular choice of the pair of opinions~$i_-$ and~$j$ from
 which it follows that
\begin{equation}
\label{eq:collision-5}
  \begin{array}{l}
    P \,(\xi_t (e) = s \,| \,\xi_{t-} (e - 1) = s_- \ \hbox{and} \ \xi_{t-} (e) = s_+ \ \hbox{and} \ x - 1 \to_t x) \vspace*{4pt} \\ \hspace*{40pt} = \
    P \,(\xi_t (e) = s \,| \,B_{t-} (i_-, j) \ \hbox{and} \ \xi_{t-} (e) = s_+ \ \hbox{and} \ x - 1 \to_t x) \end{array}
\end{equation}
 The lemma is then a direct consequence of~\eqref{eq:collision-3}--\eqref{eq:collision-5}.
\end{proof} \\ \\
 As previously mentioned, it follows from Lemma~\ref{lem:collision} that, provided the opinion model starts from a product measure in which
 the density of each opinion is constant across space and the opinion graph is distance-regular, the system of piles itself is a Markov process.
 Another important consequence is the following lemma, which gives bounds for the probabilities that the jump of an active pile onto a frozen
 pile results in a reduction or an increase of its order.
\begin{lemma} --
\label{lem:jump}
 Let~$x = e - 1/2$. Assume~\eqref{eq:uniform} and~\eqref{eq:dist-reg-1}. Then,
 $$ \begin{array}{l}
      P \,(\ceil{\xi_t (e) / \threshold} < \ceil{\xi_{t-} (e) / \threshold} \,| \,(\xi_{t-} (e - 1), \xi_{t-} (e)) = (s_-, s_+) \ \hbox{and} \ x - 1 \to_t x) \ \leq \ p_n \vspace*{4pt} \\
      P \,(\ceil{\xi_t (e) / \threshold} > \ceil{\xi_{t-} (e) / \threshold} \,| \,(\xi_{t-} (e - 1), \xi_{t-} (e)) = (s_-, s_+) \ \hbox{and} \ x - 1 \to_t x) \ \geq \ q_n \end{array} $$
 whenever~$0 < \ceil{s_- / \threshold} = 1$ and~$\ceil{s_+ / \threshold} = n > 1$.
\end{lemma}
\begin{proof}
 Let~$p (s_-, s_+, s)$ be the conditional probability
 $$ P \,(\xi_t (e) = s \,| \,(\xi_{t-} (e - 1), \xi_{t-} (e)) = (s_-, s_+) \ \hbox{and} \ x - 1 \to_t x) $$
 in the statement of Lemma~\ref{lem:collision}.
 Then, the probability that the jump of an active pile onto the pile of order~$n$ at edge~$e$ results in a reduction of its order is smaller than
\begin{equation}
\label{eq:jump-1}
  \begin{array}{l} \max \,\{\sum_{s : \ceil{s / \threshold} = n - 1} \,p (s_-, s_+, s) : \ceil{s_- / \threshold} = 1 \ \hbox{and} \ \ceil{s_+ / \threshold} = n \} \end{array}
\end{equation}
 while the probability that the jump of an active pile onto the pile of order~$n$ at edge~$e$ results in an increase of its order is larger than
\begin{equation}
\label{eq:jump-2}
  \begin{array}{l} \min \,\{\sum_{s : \ceil{s / \threshold} = n + 1} \,p (s_-, s_+, s) : \ceil{s_- / \threshold} = 1 \ \hbox{and} \ \ceil{s_+ / \threshold} = n \}. \end{array}
\end{equation}
 But according to Lemma~\ref{lem:collision}, we have
 $$ p (s_-, s_+, s) \ = \ f (s_-, s_+, s) / h (s_+) $$
 therefore~\eqref{eq:jump-1}--\eqref{eq:jump-2} are equal to~$p_n$ and~$q_n$, respectively.
\end{proof} \\ \\
 We refer to Figure~\ref{fig:coupling} for a schematic illustration of the previous lemma.
 In order to prove the theorem, we now use Lemmas~\ref{lem:collision}--\ref{lem:jump} to find a stochastic lower bound for the contribution of each edge.
 To express this lower bound, we let~$X_t$ be the discrete-time birth and death Markov chain with transition probabilities
 $$ p (n, n - 1) \ = \ p_n \qquad p (n, n) \ = \ 1 - p_n - q_n \qquad p (n, n + 1) \ = \ q_n $$
 for all~$1 < n < M := \ceil{\diameter / \threshold}$ and boundary conditions
 $$ p (1, 1) \ = \ 1 \quad \hbox{and} \quad p (M, M - 1) \ = \ 1 - p (M, M) \ = \ p_M. $$
 This process will allow us to retrace the history of a frozen pile until time~$T_e$ when it becomes an active pile.
 To begin with, we use a first-step analysis to compute explicitly the expected value of the first hitting time to state~1 of the birth and death process.
\begin{lemma} --
\label{lem:hitting}
 Let~$T_n := \inf \,\{t : X_t = n \}$. Then,
 $$ E \,(T_1 \,| \,X_0 = k) \ = \ 1 + \Weight (k) \quad \hbox{for all} \quad 0 < k \leq M = \ceil{\diameter / \threshold}. $$
\end{lemma}
\begin{proof}
 Let~$\sigma_n := E \,(T_{n - 1} \,| \,X_0 = n)$.
 Then, for all~$1 < n < M$,
 $$ \begin{array}{rcl}
    \sigma_n & = & p (n, n - 1) + (1 + \sigma_n) \,p (n, n) + (1 + \sigma_n + \sigma_{n + 1}) \,p (n, n + 1) \vspace*{3pt} \\
             & = & p_n + (1 + \sigma_n)(1 - p_n - q_n) + (1 + \sigma_n + \sigma_{n + 1}) \,q_n \vspace*{3pt} \\
             & = & p_n + (1 + \sigma_n)(1 - p_n) + q_n \,\sigma_{n + 1} \vspace*{3pt} \\
             & = & 1 + (1 - p_n) \,\sigma_n + q_n \,\sigma_{n + 1} \end{array} $$
 from which it follows, using a simple induction, that
\begin{equation}
\label{eq:hitting-1}
  \begin{array}{rcl}
    \sigma_n & = & 1 / p_n + \sigma_{n + 1} \,q_n / p_n \vspace*{4pt} \\
             & = & 1 / p_n +  q_n / (p_n \,p_{n + 1}) + \sigma_{n + 2} \,(q_n \,q_{n + 1}) / (p_n \,p_{n + 1}) \vspace*{4pt} \\
             & = & \sum_{n \leq m < M} \,(q_n \cdots q_{m - 1}) / (p_n \cdots p_m) + \sigma_M \,(q_n \cdots q_{M - 1}) / (p_n \cdots p_{M - 1}). \end{array} 
\end{equation}
 Since~$p (M, M - 1) = 1 - p (M, M) = p_M$, we also have
\begin{equation}
\label{eq:hitting-2}
  \sigma_M \ = \ E \,(T_{M - 1} \,| \,X_0 = M) \ = \ E \,(\geometric (p_M)) \ = \ 1 / p_M.
\end{equation}
 Combining~\eqref{eq:hitting-1}--\eqref{eq:hitting-2}, we deduce that
 $$ \begin{array}{l} \sigma_n \ = \ \sum_{n \leq m \leq M} \,(q_n \,q_{n + 1} \cdots q_{m - 1}) / (p_n \,p_{n + 1} \cdots p_m), \end{array} $$
 which finally gives
 $$ \begin{array}{rcl}
      E \,(T_1 \,| \,X_0 = k) & = & \sum_{1 < n \leq k} \,E \,(T_{n - 1} \,| \,X_0 = n) \ = \ \sum_{1 < n \leq k} \,\sigma_n \vspace*{4pt} \\
                              & = & \sum_{1 < n \leq k} \,\sum_{n \leq m \leq M} \,(q_n \cdots q_{m - 1}) / (p_n \cdots p_m) \ = \ 1 + \Weight (k). \end{array} $$
 This completes the proof.
\end{proof} \\ \\
 The next lemma gives a lower bound for the contribution~\eqref{eq:contribution-frozen} of an edge~$e$ that keeps track of the number
 of active piles that jump onto~$e$ before the pile at~$e$ becomes active.
 The key is to show how the number of jumps relates to the birth and death process.
 Before stating our next result, we recall that~$T_e$ is the first time the pile of particles at edge~$e$ becomes active.
\begin{figure}[t]
\centering
\scalebox{0.45}{\input{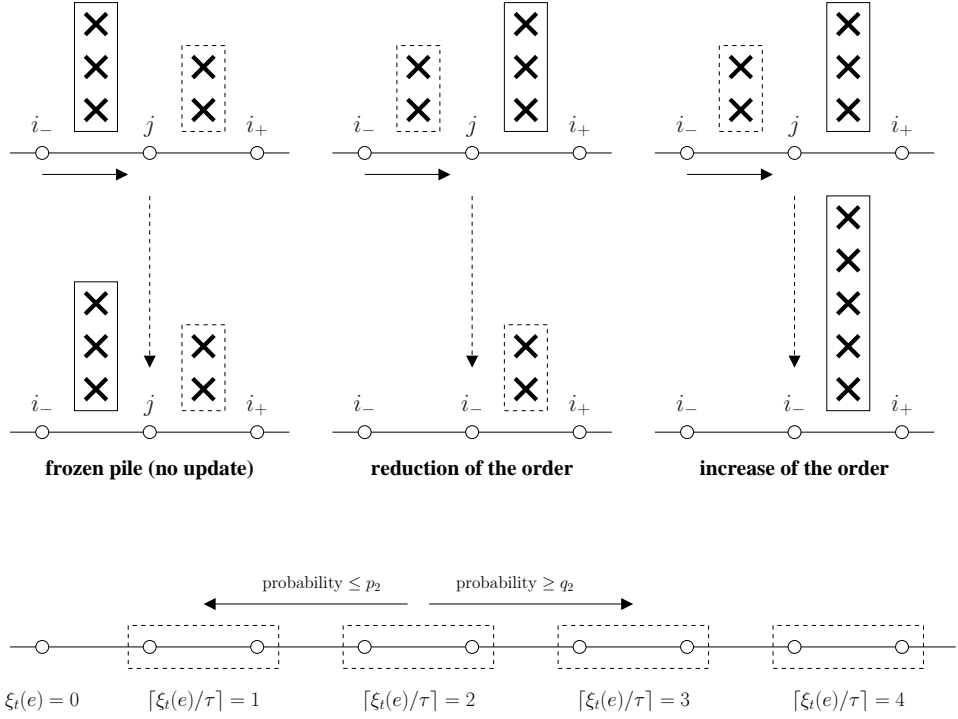}}
\caption{\upshape{Schematic illustration of the coupling between the opinion model and the system of piles along with their evolution rules.
 In our example, the threshold~$\threshold = 2$, which makes piles with three or more particles frozen piles and piles with one or two particles active piles.}}
\label{fig:coupling}
\end{figure}
\begin{lemma} --
\label{lem:coupling}
 Assume~\eqref{eq:uniform} and~\eqref{eq:dist-reg-1}.
 Then, for~$1 < k \leq \ceil{\diameter / \threshold}$,
 $$ \begin{array}{l} E \,(\cont (e \,| \,T_e < \infty)) \ \geq \ \Weight (k) \quad \hbox{when} \quad \ceil{\xi_0 (e) / \threshold} = k. \end{array} $$
\end{lemma}
\begin{proof}
 Since active piles have at most~$\threshold$ particles, the triangle inequality~\eqref{eq:triangle} implies that the jump of an active pile onto
 a frozen pile can only increase or decrease its size by at most~$\threshold$ particles, and therefore can only increase or decrease its order by at most one.
 In particular,
 $$ \begin{array}{l}
      P \,(|\ceil{\xi_t (e) / \threshold} - \ceil{\xi_{t-} (e) / \threshold}| > 2 \,| \,x - 1 \to_t x) \ = \ 0. \end{array} $$
 This, together with the bounds in Lemma~\ref{lem:jump} and the fact that the outcomes of consecutive jumps of active piles onto a frozen pile are independent
 as explained in the proof of Lemma~\ref{lem:collision}, implies that the order of a frozen pile before it becomes active dominates stochastically the
 state of the birth and death process~$X_t$ before it reaches state~1.
 In particular,
 $$ \begin{array}{l} E \,(\cont (e \,| \,T_e < \infty)) \ \geq \ - 1 + E \,(T_1 \,| \,X_0 = k) \quad \hbox{when} \quad \ceil{\xi_0 (e) / \threshold} = k. \end{array} $$
 Using Lemma~\ref{lem:hitting}, we conclude that
 $$ \begin{array}{l} E \,(\cont (e \,| \,T_e < \infty)) \ \geq \ - 1 + (1 + \Weight (k)) \ = \ \Weight (k) \end{array} $$
 whenever~$\ceil{\xi_0 (e) / \threshold} = k$.
\end{proof} \\ \\
 We now have all the necessary tools to prove the theorem.
 The key idea is the same as in the proof of Lemma~\ref{lem:expected-weight} but relies on the previous lemma in place of Lemma~\ref{lem:deterministic}. \\ \\
\begin{demo}{Theorem~\ref{th:dist-reg}} --
 Assume~\eqref{eq:uniform} and~\eqref{eq:dist-reg-1} and
 $$ \begin{array}{l} S_{\reg} (\OO, \threshold) \ = \ \sum_{k > 0} \,(\Weight (k) \,\sum_{s : \ceil{s / \threshold} = k} \,h (s)) \ > \ 0. \end{array} $$
 Since the opinion graph is distance-regular,
 $$ \begin{array}{rcl}
      P \,(\xi_0 (e) = s) & = & \sum_{i \in \VO} \,P \,(\xi_0 (e) = s \,| \,\eta_0 (e - 1/2) = i) \,P \,(\eta_0 (e - 1/2) = i) \vspace*{4pt} \\
                          & = & \sum_{i \in \VO} \,F^{-1} \,\card \{j \in \VO : d (i, j) = s \} \ P \,(\eta_0 (e - 1/2) = i) \vspace*{4pt} \\
                          & = & \sum_{i \in \VO} \,F^{-1} \,h (s) \,P \,(\eta_0 (e - 1/2) = i) \ = \ F^{-1} \,h (s). \end{array} $$
 Using also Lemma~\ref{lem:coupling}, we get
 $$ \begin{array}{rcl}
      E \,(\cont (e \,| \,T_e < \infty)) & \geq & \sum_{k > 0} \,\Weight (k) \,P \,(\ceil{\xi_0 (e) / \threshold} = k) \vspace*{4pt} \\
                                         & = & \sum_{k > 0} \,\Weight (k) \,P \,((k - 1) \threshold < \xi_0 (e) \leq k \threshold) \vspace*{4pt} \\
                                         & = & \sum_{k > 0} \,\Weight (k) \,\sum_{s : \ceil{s / \threshold} = k} \,F^{-1} \,h (s) \vspace*{4pt} \\
                                         & = & F^{-1} \,S_{\reg} (\OO, \threshold) \ > \ 0. \end{array} $$
 Now, let~$\Weight_e$ be the collection of random variables
 $$ \begin{array}{l} \Weight_e \ := \ \sum_{k > 0} \,\Weight (k) \,\ind \{\xi_0 (e) = k \} \quad \hbox{for all} \quad e \in \Z + 1/2. \end{array} $$
 Using Lemma~\ref{lem:weight} and the fact the number of collisions to turn a frozen pile into an active pile is independent for different
 frozen piles, we deduce that there exists~$c_{11} > 0$ such that
 $$ \begin{array}{l}
      P \,(\sum_{e \in (0, N)} \cont (e \,| \,T_e < \infty) \leq 0) \ \leq \
      P \,(\sum_{e \in (0, N)} \Weight_e \leq 0) \vspace*{4pt} \\ \hspace*{40pt} = \
      P \,(\sum_{e \in (0, N)} \,(\Weight_e - E \Weight_e) \notin (- \ep N, \ep N)) \ \leq \ \exp (- c_{11} N) \end{array} $$
 for all~$N$ large.
 This, together with~\eqref{eq:inclusion-1}, implies that
 $$ \begin{array}{rcl}
      P \,(H_N) & \leq &
      P \,(\sum_{e \in (l, r)} \cont (e \,| \,T_e < \infty) \leq 0 \ \hbox{for some~$l < - N$ and~$r \geq 0$}) \vspace*{4pt} \\ & \leq &
           \sum_{l < - N} \,\sum_{r \geq 0} \,\exp (- c_{11} \,(r - l)) \ \to \ 0 \end{array} $$
 as~$N \to \infty$.
 In particular, it follows from Lemma~\ref{lem:fixation-condition} that the process fixates.
\end{demo}

%%%%%%%%%%%%%%%%%%%%%%%%%%%%%%%%%%%%%%%%%%%%%%%%%%%%%%%%%%%%%%%%%%%%%%%%%%%%%%%%%%%%%%%%%%%%%%%%%%%%%%%%%%%%%%%%%%%%%%%%%%%%%%%%%%%%%%%%%%

\section{Proof of Corollaries~\ref{cor:path}--\ref{cor:hypercube}}
\label{sec:graphs}

\indent This section is devoted to the proof of Corollaries~\ref{cor:path}--\ref{cor:hypercube} that give sufficient conditions
 for fluctuation and fixation of the infinite system for the opinion graphs shown in Figure~\ref{fig:graphs}.
 To begin with, we prove the fluctuation part of all the corollaries at once. \\ \\
\begin{demo}{Corollaries~\ref{cor:path}--\ref{cor:hypercube} (fluctuation)} --
 We start with the tetrahedron.
 In this case, the diameter equals one therefore, whenever the threshold is positive, the system reduces to a four-opinion voter model,
 which is known to fluctuate according to~\cite{arratia_1983}.
 To deal with paths and stars, we recall that combining Theorem~\ref{th:fluctuation}a and Lemma~\ref{lem:partition} gives fluctuation
 when~$\radius \leq \threshold$.
 Recalling also the expression of the radius from Table~\ref{tab:summary} implies fluctuation when
 $$ \begin{array}{rl}
      F \leq 2 \threshold + 1 & \hbox{for the path with~$F$ vertices} \vspace*{3pt} \\
      r \leq   \threshold     & \hbox{for the star with~$b$ branches of length~$r$}. \end{array} $$
 For the other graphs, it suffices to find a partition that satisfies~\eqref{eq:fluctuation}.
 For the remaining four regular polyhedra and the hypercubes, we observe that there is a unique vertex at distance~$\diameter$ of any
 given vertex.
 In particular, fixing an arbitrary vertex~$i_-$ and setting
 $$ V_1 \ := \ \{i_-, i_+ \} \quad \hbox{and} \quad V_2 \ := \ V \setminus V_1 \quad \hbox{where} \quad d (i_-, i_+) = \diameter $$
 defines a partition of the set of opinions such that
 $$ d (i, j) \ \leq \ \diameter - 1 \quad \hbox{for all} \quad (i, j) \in V_1 \times V_2. $$
 Recalling the expression of the diameter from Table~\ref{tab:summary} and using~Theorem~\ref{th:fluctuation}a give the fluctuation
 parts of Corollaries~\ref{cor:polyhedron} and~\ref{cor:hypercube}.
 Using the exact same approach implies fluctuation when the opinion graph is a cycle with an even number of vertices and~$F \leq 2 \threshold + 2$.
 For cycles with an odd number of vertices, we again use Lemma~\ref{lem:partition} to deduce fluctuation if
 $$ \integer{F / 2} = \radius \leq \threshold \quad \hbox{if and only if} \quad F \leq 2 \threshold + 1 \quad \hbox{if and only if} \quad F \leq 2 \threshold + 2, $$
 where the last equivalence is true because~$F$ is odd.
\end{demo} \\ \\
 We now prove the fixation part of the corollaries using Theorems~\ref{th:fixation} and~\ref{th:dist-reg}.
 The first two classes of graphs, paths and stars, are not distance-regular therefore, to study the behavior of the
 model for these opinion graphs, we rely on the first part of Theorem~\ref{th:fixation}. \\ \\
\begin{demo}{Corollary~\ref{cor:path} (path)} --
 Assume that~$4 \threshold < \diameter = F - 1 \leq 5 \threshold$. Then,
 $$ \begin{array}{rcl}
     S (\OO, \threshold) & = & \sum_{k > 0} \,((k - 2) \,\sum_{s : \ceil{s / \threshold} = k} \,N (\OO, s)) \vspace*{4pt} \\
                            & = & \sum_{0 < k \leq 4} \,((k - 2) \,\sum_{s : \ceil{s / \threshold} = k} \,2 \,(F - s)) + 3 \,\sum_{4 \threshold < s \leq d} \,2 \,(F - s) \vspace*{4pt} \\
                            & = & \sum_{0 < k \leq 4} \,((k - 2)(2 F \threshold - (k \threshold)(k \threshold + 1) + ((k - 1) \,\threshold)((k - 1) \,\threshold + 1)) \vspace*{4pt} \\ && \hspace*{50pt} + \
                                   3 \,(2F \,(F - 4 \threshold - 1) - F \,(F - 1) + 4 \threshold \,(4 \threshold + 1)) \vspace*{4pt} \\
                            & = & 4 F \threshold + \threshold \,(\threshold + 1) + 2 \threshold \,(2 \threshold + 1) + 3 \threshold \,(3 \threshold + 1) \vspace*{4pt} \\ && \hspace*{50pt} + \
                                  4 \threshold \,(4 \threshold + 1) + 6 F \,(F - 4 \threshold - 1) - 3 F \,(F - 1) \vspace*{4pt} \\
                            & = & 3 F^2 - (20 \threshold + 3) \,F + 10 \,(3 \threshold + 1) \,\threshold. \end{array} $$
 Since the largest root~$F_+ (\threshold)$ of this polynomial satisfies
 $$ 4 \threshold \leq F_+ (\threshold) - 1 = (1/6)(20 \,\threshold + 3 + \sqrt{40 \,\threshold^2 + 9}) - 1 \leq 5 \threshold \quad \hbox{for all} \quad \threshold \geq 1 $$
 and since for any fixed~$\threshold$ the function~$F \mapsto S (\OO, \threshold)$ is nondecreasing, we deduce that fixation occurs under the assumptions of the lemma
 according to Theorem~\ref{th:fixation}.
\end{demo} \\ \\
 The case of the star with~$b$ branches of equal length~$r$ is more difficult mainly because there are two different expressions for the number of pairs of vertices at a
 given distance of each other depending on whether the distance is smaller or larger than the branches' length.
 In the next lemma, we compute the number of pairs of vertices at a given distance of each other, which we then use to find a condition for fixation when
 the opinion graph is a star.
\begin{lemma} --
\label{lem:star}
 For the star with~$b$ branches of length~$r$,
 $$ \begin{array}{rclcl}
      N (\OO, s) & = & b \,(2r + (b - 3)(s - 1)) & \hbox{for all} & s \in (0, r] \vspace*{3pt} \\
                    & = & b \,(b - 1)(2r - s + 1)   & \hbox{for all} & s \in (r, 2r]. \end{array} $$
\end{lemma}
\begin{proof}
 Let~$n_1 (s)$ and~$n_2 (s)$ be respectively the number of directed paths of length~$s$ embedded in a given branch of the star and the total
 number of directed paths of length~$s$ embedded in a given pair of branches of the star.
 Then, as in the proof of the corollary for paths,
 $$ n_1 (s) = 2 \,(r + 1 - s) \quad \hbox{and} \quad n_2 (s) = 2 \,(2r + 1 - s) \quad \hbox{for all} \quad s \leq r. $$
 Since there are~$b$ branches and $(1/2)(b - 1) \,b$ pairs of branches, and since self-avoiding paths embedded in the star cannot
 intersect more than two branches, we deduce that
 $$ \begin{array}{rcl}
      N (\OO, s) & = & b \,n_1 (s) + ((1/2)(b - 1) \,b)(n_2 (s) - 2 n_1 (s)) \vspace*{4pt} \\
                    & = & 2b \,(r + 1 - s) + b \,(b - 1)(s - 1) \vspace*{4pt} \\
                    & = & b \,(2r + 2 \,(1 - s) + (b - 1)(s - 1)) \ = \ b \,(2r + (b - 3)(s - 1)) \end{array} $$
 for all~$s \leq r$.
 To deal with~$s > r$, we let~$o$ be the center of the star and observe that there is no vertex at distance~$s$ of vertices which are
 close to the center whereas there are~$b - 1$ vertices at distance~$s$ from vertices which are far from the center.
 More precisely,
 $$ \begin{array}{rclcl}
    \card \{j \in \VO : d (i, j) = s \} & = & 0     & \quad \hbox{when} & d (i, o) < s - r \vspace*{3pt} \\
    \card \{j \in \VO : d (i, j) = s \} & = & b - 1 & \quad \hbox{when} & d (i, o) \geq s - r. \end{array} $$
 The number of directed paths of length~$s$ is then given by
 $$ \begin{array}{rcl}
      N (\OO, s) & = & (b - 1) \,\card \{i \in \VO : d (i, o) \geq s - r \} \vspace*{4pt} \\
                    & = & b \,(b - 1)(r - (s - r - 1)) \ = \ b \,(b - 1)(2r - s + 1) \end{array} $$
 for all~$s > r$.
 This completes the proof of the lemma.
\end{proof} \\ \\
\begin{demo}{Corollary~\ref{cor:star} (star)} --
 Assume that~$3 \threshold < \diameter = 2r \leq 4 \threshold$. Then,
 $$ \begin{array}{rcl}
     S (\OO, \threshold) & = & \sum_{k > 0} \,((k - 2) \,\sum_{s : \ceil{s / \threshold} = k} \,N (\OO, s)) \vspace*{4pt} \\
                            & = & - \ \sum_{0 < s \leq \threshold} \,N (\OO, s) + \sum_{2 \threshold < s \leq 3 \threshold} \,N (\OO, s) + 2 \,\sum_{3 \threshold < s \leq 2r} \,N (\OO, s). \end{array} $$
 Since~$\threshold < r \leq 2 \threshold$, it follows from Lemma~\ref{lem:star} that
 $$ \begin{array}{rcl}
     S (\OO, \threshold) & = & - \ \sum_{0 < s \leq \threshold} \,b \,(2r + (b - 3)(s - 1)) \vspace*{4pt} \\ &&
                                  + \ \sum_{2 \threshold < s \leq 3 \threshold} \,b \,(b - 1)(2r - s + 1) + 2 \,\sum_{3 \threshold < s \leq 2r} \,b \,(b - 1)(2r - s + 1) \vspace*{4pt} \\
                            & = & - \ b \,(2r - b + 3) \,\threshold - (b/2)(b - 3) \,\threshold \,(\threshold + 1) \vspace*{4pt} \\ &&
                                  + \ b \,(b - 1)(2r + 1) \,\threshold + (b/2)(b - 1)(2 \threshold \,(2 \threshold + 1) - 3 \threshold \,(3 \threshold + 1)) \vspace*{4pt} \\ &&
                                  + \ 2b \,(b - 1)(2r + 1)(2r - 3 \threshold) + b \,(b - 1)(3 \threshold \,(3 \threshold + 1) - 2 r \,(2r + 1)). \end{array} $$
 Expanding and simplifying, we get
 $$ (1/b) \,S (\OO, \threshold) \ = \ 4 \,(b - 1) \,r^2 + 2 \,((4 - 5b) \,\threshold + b - 1) \,r + (6b - 5) \,\threshold^2 + (1 - 2b) \,\threshold. $$
 As for paths, the result is a direct consequence of Theorem~\ref{th:fixation}.
\end{demo} \\ \\
 The remaining graphs in Figure~\ref{fig:graphs} are distance-regular, which makes Theorem~\ref{th:dist-reg} applicable.
 Note that the conditions for fixation in the last three corollaries give minimal values for the confidence threshold that lie between one third and
 one half of the diameter.
 In particular, we apply the theorem in the special case when~$\ceil{\diameter / \threshold} = 3$.
 In this case, we have
 $$ \Weight (1) \ = \ - 1 \qquad \Weight (2) \ = \ \Weight (1) + (1 / p_2)(1 + q_2 / p_3) \qquad \Weight (3) \ = \ \Weight + 1 / p_3 $$
 so the left-hand side of~\eqref{eq:th-dist-reg} becomes
\begin{equation}
\label{eq:common}
  \begin{array}{rcl}
    S_{\reg} (\OO, \threshold) & = &   \sum_{0 < k \leq 3} \,(\Weight (k) \,\sum_{s : \ceil{s / \threshold} = k} \,h (s)) \vspace*{4pt} \\ & = &
                                     - \ (h (1) + h (2) + \cdots + h (\diameter)) \vspace*{4pt} \\ &&
                                     + \ (1/p_2)(1 + q_2 / p_3)(h (\threshold + 1) + h (\threshold + 2) + \cdots + h (\diameter)) \vspace*{4pt} \\ &&
                                     + \ (1/p_3)(h (2 \threshold + 1) + h (2 \threshold + 2) + \cdots + h (\diameter)). \end{array}
\end{equation}
 This expression is used repeatedly to prove the remaining corollaries. \\ \\
\begin{demo}{Corollary~\ref{cor:polyhedron} (cube)} --
 When~$\OO$ is the cube and~$\threshold = 1$, we have
 $$ p_2 \ = \ f (1, 2, 1) / h (2) \ = \ 2/3 \quad \hbox{and} \quad q_2 \ = \ f (1, 2, 3) / h (2) \ = \ 1/3 $$
 which, together with~\eqref{eq:common} and the fact that~$p_3 \leq 1$, implies that
 $$ \begin{array}{rcl}
      S_{\reg} (\OO, 1) & \geq & - \ (h (1) + h (2) + h (3)) + (1/p_2)(1 + q_2)(h (2) + h (3)) + h (3) \vspace*{4pt} \\
                           & = & - \ (3 + 3 + 1) + (3/2)(1 + 1/3)(3 + 1) + 1 \ = \ 2 \ > \ 0. \end{array} $$
 This proves fixation according to Theorem~\ref{th:dist-reg}.
\end{demo} \\ \\
\begin{demo}{Corollary~\ref{cor:polyhedron} (icosahedron)} --
 When~$\OO$ is the icosahedron and~$\threshold = 1$,
 $$ p_2 \ = \ f (1, 2, 1) / h (2) \ = \ 2/5 \qquad \hbox{and} \qquad q_2 \ = \ f (1, 2, 3) / h (2) \ = \ 1/5. $$
 Using in addition~\eqref{eq:common} and the fact that~$p_3 \leq 1$, we obtain
 $$ \begin{array}{rcl}
      S_{\reg} (\OO, 1) & \geq & - \ (h (1) + h (2) + h (3)) + (1/p_2)(1 + q_2)(h (2) + h (3)) + h (3) \vspace*{4pt} \\
                           & = & - \ (5 + 5 + 1) + (5/2)(1 + 1/5)(5 + 1) + 1 \ = \ 8 \ > \ 0 \end{array} $$
 which, according to Theorem~\ref{th:dist-reg}, implies fixation.
\end{demo} \\ \\
\begin{demo}{Corollary~\ref{cor:polyhedron} (dodecahedron)} --
 Fixation of the opinion model when the threshold equals one directly follows from Theorem~\ref{th:fixation} since in this case
 $$ \begin{array}{rcl}
      F^{-1} \,S (\OO, 1) & = & (1/20)(- h (1) + h (3) + 2 \,h (4) + 3 \,h (5)) \vspace*{3pt} \\
                             & = & (1/20)(- 3 + 6 + 2 \times 3 + 3 \times 1) \ = \ 3/5 \ > \ 0. \end{array} $$
 However, when the threshold~$\threshold = 2$,
 $$ \begin{array}{rcl}
      F^{-1} \,S (\OO, 2) & = & (1/20)(- h (1) - h (2) + h (5)) \vspace*{3pt} \\
                             & = & (1/20)(- 3 - 6 + 1) \ = \ - 2/5 \ < \ 0 \end{array} $$
 so we use Theorem~\ref{th:dist-reg} instead: when~$\threshold = 2$, we have
 $$ \begin{array}{rcl}
      p_2 & = & \max \,\{\sum_{s = 1, 2} f (s_-, s_+, s) / h (s_+) : s_- = 1, 2 \ \hbox{and} \ s_+ = 3, 4 \} \vspace*{4pt} \\
          & = & \max \,\{f (1, 3, 2) / h (3), (f (2, 3, 2) + f (2, 3, 1)) / h (3), f (2, 4, 2) / h (4) \} \vspace*{4pt} \\
          & = & \max \,\{2/6, (2 + 1) / 6, 1/3 \} \ = \ 1/2. \end{array} $$
 In particular, using~\eqref{eq:common} and the fact that~$p_3 \leq 1$ and~$q_2 \geq 0$, we get
 $$ \begin{array}{rcl}
      S_{\reg} (\OO, 2) & \geq & - \ (h (1) + h (2) + h (3) + h (4) + h (5)) \vspace*{4pt} \\ && \hspace*{25pt} + \
                                     (1/p_2)(h (3) + h (4) + h (5)) + h (5) \vspace*{4pt} \\
                           & = & - \ (3 + 6 + 6 + 3 + 1) + 2 \times (6 + 3 + 1) + 1 \ = \ 2 \ > \ 0, \end{array} $$
 which again gives fixation.
\end{demo} \\ \\
\begin{demo}{Corollary~\ref{cor:cycle} (cycle)} --
 Regardless of the parity of~$F$,
\begin{equation}
\label{eq:cycle-1}
  \begin{array}{rclclcl}
     f (s_-, s_+, s) & = & 0 & \hbox{when} & s_- \leq s_+ \leq \diameter & \hbox{and} & s > s_+ - s_- \vspace*{2pt} \\
     f (s_-, s_+, s) & = & 1 & \hbox{when} & s_- \leq s_+ \leq \diameter & \hbox{and} & s = s_+ - s_- \end{array}
\end{equation}
 while the number of vertices at distance~$s_+$ of a given vertex is
\begin{equation}
\label{eq:cycle-2}
  h (s_+) = 2 \ \ \hbox{for all} \ \ s_+ < F/2 \quad \hbox{and} \quad h (s_+) = 1 \ \ \hbox{when} \ \ s_+ = F/2 \in \N.
\end{equation}
 Assume that~$F = 4 \threshold + 2$.
 Then, $\diameter = 2 \threshold + 1$ so it follows from~\eqref{eq:cycle-1}--\eqref{eq:cycle-2} that
 $$ \begin{array}{rcl}
      p_2 & = & \max \,\{\sum_{s : \ceil{s / \threshold} = 1} f (s_-, s_+, s) / h (s_+) : \ceil{s_- / \threshold} = 1 \ \hbox{and} \ \ceil{s_+ / \threshold} = 2 \} \vspace*{3pt} \\
          & = & \max \,\{f (s_-, s_+, s_+ - s_-) / h (s_+) : \ceil{s_- / \threshold} = 1 \ \hbox{and} \ \ceil{s_+ / \threshold} = 2 \} \vspace*{3pt} \\
          & = & \max \,\{f (s_-, s_+, s_+ - s_-) / h (s_+) : \ceil{s_+ / \threshold} = 2 \} \ = \ 1/2. \end{array} $$
 Using in addition that~$p_3 \leq 1$ and~$q_2 \geq 0$ together with~\eqref{eq:common}, we get
 $$ \begin{array}{rcl}
      S_{\reg} (\OO, \threshold) & \geq & - \ (h (1) + h (2) + \cdots + h (2 \threshold + 1)) \vspace*{4pt} \\ &&
                                             + \ (1/p_2)(h (\threshold + 1) + h (\threshold + 2) + \cdots + h (2 \threshold + 1)) + h (2 \threshold + 1) \vspace*{4pt} \\
                                    & = & - \ (4 \threshold + 1) + 2 \times (2 \threshold + 1) + 1 \ = \ 2 \ > \ 0. \end{array} $$
 In particular, the corollary follows from Theorem~\ref{th:dist-reg}.
\end{demo} \\ \\
\begin{demo}{corollary~\ref{cor:hypercube} (hypercube)} --
 The first part of the corollary has been explained heuristically in~\cite{adamopoulos_scarlatos_2012}.
 To turn it into a proof, we first observe that opinions on the hypercube can be represented by vectors with coordinates equal to zero or one while the distance
 between two opinions is the number of coordinates the two corresponding vectors disagree on.
 In particular, the number of opinions at distance~$s$ of a given opinion, namely~$h (s)$, is equal to the number of subsets of size~$s$ of a set of size~$d$.
 Therefore, we have the symmetry property
\begin{equation}
\label{eq:hypercube-1}
  h (s) \ = \ {d \choose s} \ = \ {d \choose d - s} \ = \ h (d - s) \quad \hbox{for} \quad s = 0, 1, \ldots, d,
\end{equation}
 from which it follows that, for~$d = 3 \threshold + 1$,
 $$ \begin{array}{rcl}
      2^{-d} \,S (\OO, \threshold) & = & - \ h (1) - \cdots - h (\threshold) + h (2 \threshold + 1) + \cdots + h (d - 1) + 2 \,h (d) \vspace*{3pt} \\
                                   & = & h (d - 1) - h (1) + h (d - 2) - h (2) + \cdots + h (d - \threshold) - h (\threshold) + 2 \,h (d) \vspace*{3pt} \\
                                   & = & 2 \,h (d) \ = \ 2 \ > \ 0. \end{array} $$
 Since in addition the function~$d \mapsto S (\OO, \threshold)$ is nondecreasing, a direct application of Theorem~\ref{th:fixation} gives the first part of the corollary.
 The second part is more difficult.
 Note that, to prove this part, it suffices to show that, for any fixed~$\sigma > 0$, fixation occurs when
\begin{equation}
\label{eq:hypercube-2}
  d \ = \ (2 + 3 \sigma) \,\threshold \quad \hbox{and} \quad \threshold \ \ \hbox{is large}.
\end{equation}
 The main difficulty is to find a good upper bound for~$p_2$ which relies on properties of the hypergeometric random variable.
 Let~$u$ and~$v$ be two opinions at distance~$s_-$ of each other.
 By symmetry, we may assume without loss of generality that both vectors disagree on their first~$s_-$ coordinates.
 Then, changing each of the first~$s_-$ coordinates in either one vector or the other vector and changing each of the remaining coordinates in either both vectors
 simultaneously or none of the vectors result in the same vector.
 In particular, choosing a vector~$w$ such that
 $$ d (u, w) \ = \ s_+ \quad \hbox{and} \quad d (v, w) \ = \ s $$
 is equivalent to choosing~$a$ of the first~$s_-$ coordinates and then choosing~$b$ of the remaining~$d - s_-$ coordinates with the following constraint:
 $$ a + b \ = \ s_+ \quad \hbox{and} \quad (s_- - a) + b \ = \ s. $$
 In particular, letting~$K := \ceil{(1/2)(s_- + s_+ - \threshold)}$, we have
 $$ \sum_{s = 1}^{\threshold} \ f (s_-, s_+, s) \ = \ \sum_{a = K}^{s_-} {s_- \choose a}{d - s_- \choose s_+ - a} \ = \ h (s_+) \,P \,(Z \geq K) $$
 where~$Z = \hypergeometric (d, s_-, s_+)$.
 In order to find an upper bound for~$p_2$ and deduce fixation, we first prove the following lemma about the hypergeometric random variable.
\begin{lemma} --
 Assume~\eqref{eq:hypercube-2}, that~$\ceil{s_- / \threshold} = 1$ and~$\ceil{s_+ / \threshold} = 2$. Then,
 $$ P \,(Z \geq K) \ = \ \sum_{a = K}^{s_-} {s_- \choose a}{d - s_- \choose s_+ - a}{d \choose s_+}^{-1} \leq \ 1/2. $$
\end{lemma}
\begin{proof}
 The proof is made challenging by the fact that there is no explicit expression for the cumulative distribution function of the hypergeometric random variable
 and the idea is to use a combination of symmetry arguments and large deviation estimates.
 Symmetry is used to prove the result when~$s_-$ is small while large deviation estimates are used for larger values.
 Note that the result is trivial when~$s_+ > s_- + \threshold$ since in this case the sum in the statement of the lemma is empty so equal to zero.
 To prove the result when the sum is nonempty, we distinguish two cases. \vspace*{5pt} \\
{\bf Small active piles} -- Assume that~$s_- < \sigma \threshold$. Then,
\begin{equation}
\label{eq:hypergeometric-1}
  \begin{array}{rcl}
    s_+ & \leq & s_- + \threshold < (1 + \sigma) \,\threshold \ = \ (1/2)(d - \sigma \threshold) \ < \ (1/2)(d - s_-) \vspace*{3pt} \\
      K & \geq & (1/2)(s_- + s_+ - \threshold) \ > \ s_- / 2 \ > \ s_- - K \end{array}
\end{equation}
 from which it follows that
\begin{equation}
\label{eq:hypergeometric-2}
  {s_- \choose a}{d - s_- \choose s_+ - a} \ \leq \ {s_- \choose a}{d - s_- \choose s_+ - s_- + a} \quad \hbox{for all} \quad K \leq a \leq s_-.
\end{equation}
 Using~\eqref{eq:hypergeometric-2} and again the second part of~\eqref{eq:hypergeometric-1}, we deduce that
 $$ \begin{array}{rcl}
      h (s_+) \,P \,(Z \geq K) & = & \displaystyle \sum_{a = K}^{s_-} {s_- \choose a}{d - s_- \choose s_+ - a}
                               \ \leq \ \displaystyle \sum_{a = K}^{s_-} {s_- \choose a}{d - s_- \choose s_+ - s_- + a} \vspace*{4pt} \\
                               & = & \displaystyle \sum_{a = 0}^{s_- - K} {s_- \choose s_- - a}{d - s_- \choose s_+ - a}
                               \ \leq \ \displaystyle \sum_{a = 0}^{K - 1} {s_- \choose a}{d - s_- \choose s_+ - a}. \end{array} $$
 In particular, we have~$P \,(Z \geq K) \leq P \,(Z < K)$, which gives the result. \vspace*{5pt} \\
{\bf Larger active piles} -- Assume that~$\sigma \threshold \leq s_- \leq \threshold$.
 In this case, the result is a consequence of the following large deviation estimates for the hypergeometric random variable:
\begin{equation}
\label{eq:hypergeometric-3}
  P \,\bigg(Z \geq \bigg(\frac{s_-}{d} + \ep \bigg) \,s_+ \bigg) \ \leq \ \bigg(\bigg(\frac{s_-}{s_- + \ep d} \bigg)^{s_- / d + \ep} \bigg(\frac{d - s_-}{d - s_- - \ep d} \bigg)^{1 - s_- / d - \ep} \bigg)^{s_+}
\end{equation}
 for all~$0 < \ep < 1 - s_- / d$, that can be found in~\cite{hoeffding_1963}.
 Note that
 $$ \begin{array}{rcl}
      d \,(s_+ + s_- - \threshold) - 2 s_+ \,s_- & = & (d - 2 s_-) \,s_+ + d \,(s_- - \threshold) \vspace*{3pt} \\
                                                 & \geq & (d - 2 s_-)(\threshold + 1) + d \,(s_- - \threshold) \ \geq \ (d - 2 \threshold) \,s_- \vspace*{3pt} \\
                                                 & = & 3 \sigma \threshold s_- \ = \ (3 \sigma \threshold / 2 s_+)(2 s_+ \,s_-) \ \geq \ (3 \sigma / 4)(2 s_+ \,s_-) \end{array} $$
 for all~$\threshold < s_+ \leq 2 \threshold$.
 It follows that
 $$ K \ \geq \ \frac{s_+ + s_- - \threshold}{2} \ \geq \ \bigg(1 + \frac{3 \sigma}{4} \bigg) \,\frac{s_+ \,s_-}{d} \ = \ \bigg(\frac{s_-}{d} + \frac{3 \sigma s_-}{4d} \bigg) \,s_+ \ \geq \ \bigg(\frac{s_-}{d} + \frac{\sigma^2}{3} \bigg) \,s_+ $$
 which, together with~\eqref{eq:hypergeometric-3} for~$\ep = \sigma^2 / 3$, gives
 $$ \begin{array}{rcl}
      P \,(Z \geq K) & \leq &
        \displaystyle P \,\bigg(Z \geq \bigg(\frac{s_-}{d} + \ep \bigg) \,s_+ \bigg) \ \leq \ \bigg(\frac{s_-}{s_- + \ep d} \bigg)^{s_+ s_- / d} \vspace*{8pt} \\ & \leq &
        \displaystyle \bigg(\frac{3 s_-}{3 s_- + \sigma^2 d} \bigg)^{s_+ s_- / d} \leq \ \bigg(\frac{3}{3 + 2 \sigma^2} \bigg)^{(\sigma / 3) \,s_+} \leq \ \bigg(\frac{3}{3 + 2 \sigma^2} \bigg)^{(\sigma / 3) \,\threshold}. \end{array} $$
 Since this tends to zero as~$\threshold \to \infty$, the proof is complete.
\end{proof} \\ \\
 It directly follows from the lemma that
 $$ \begin{array}{l}
      p_2 \ = \ \max \,\{\sum_{s : \ceil{s / \threshold} = 1} f (s_-, s_+, s) / h (s_+) : \ceil{s_- / \threshold} = 1 \ \hbox{and} \ \ceil{s_+ / \threshold} = 2 \} \ \leq \ 1/2. \end{array} $$
 This, together with~\eqref{eq:common} and~$p_3 \leq 1$ and~$q_2 \geq 0$, implies that
 $$ \begin{array}{rcl}
     S_{\reg} (\OO, \threshold) & \geq & - \ h (1) - \cdots - h (d) + (1/p_2) \,h (\threshold + 1) + \cdots + (1/p_2) \,h (d) \vspace*{3pt} \\
                                   & \geq & - \ h (1) - \cdots - h (d) +  2 \,h (\threshold + 1) + \cdots + 2 \,h (d) \vspace*{3pt} \\
                                   & = & - \ h (1) - \cdots - h (\threshold) +  h (\threshold + 1) + \cdots + h (d). \end{array} $$
 Finally, using again~\eqref{eq:hypercube-1} and the fact that~$d > 2 \threshold$, we deduce that
 $$ \begin{array}{rcl}
     S_{\reg} (\OO, \threshold) & \geq & - \ h (1) - \cdots - h (\threshold) +  h (\threshold + 1) + \cdots + h (d) \vspace*{3pt} \\
                                   & \geq & h (d - 1) - h (1) + h (d - 2) - h (2) + \cdots + h (d - \threshold) - h (\threshold) + h (d) \vspace*{3pt} \\
                                   & = & h (d) \ = \ 1 \ > \ 0. \end{array} $$
 The corollary follows once more from Theorem~\ref{th:dist-reg}.
\end{demo}

%%%%%%%%%%%%%%%%%%%%%%%%%%%%%%%%%%%%%%%%%%%%%%%%%%%%%%%%%%%%%%%%%%%%%%%%%%%%%%%%%%%%%%%%%%%%%%%%%%%%%%%%%%%%%%%%%%%%%%%%%%%%%%%%%%%%%%%%%%

\end{document}